\documentclass[11pt]{article}
\usepackage{latexsym,amsfonts,amssymb,amsmath,amsthm}
\usepackage{graphicx}

\usepackage[usenames,dvipsnames]{color}
\usepackage{ulem}

\usepackage{color}
\begin{document}

\newtheorem{tm}{Theorem}[section]
\newtheorem{prop}[tm]{Proposition}
\newtheorem{defin}[tm]{Definition} % definition numbers are dependent on theorem numbers
\newtheorem{coro}[tm]{Corollary}
\newtheorem{lem}[tm]{Lemma}
\newtheorem{assumption}[tm]{Assumption}
\newtheorem{rk}[tm]{Remark}
\newtheorem{nota}[tm]{Notation}
\numberwithin{equation}{section}

\newcommand{\stk}[2]{\stackrel{#1}{#2}}
\newcommand{\dwn}[1]{{\scriptstyle #1}\downarrow}
\newcommand{\upa}[1]{{\scriptstyle #1}\uparrow}
\newcommand{\nea}[1]{{\scriptstyle #1}\nearrow}
\newcommand{\sea}[1]{\searrow {\scriptstyle #1}}
\newcommand{\csti}[3]{(#1+1) (#2)^{1/ (#1+1)} (#1)^{- #1
 / (#1+1)} (#3)^{ #1 / (#1 +1)}}
\newcommand{\RR}[1]{\mathbb{#1}}

\newcommand{\rd}{{\mathbb R^d}}
\newcommand{\ep}{\varepsilon}
\newcommand{\rr}{{\mathbb R}}
\newcommand{\alert}[1]{\fbox{#1}}
\newcommand{\eqd}{\sim}
\def\p{\partial}
\def\R{{\mathbb R}}
\def\N{{\mathbb N}}
\def\Q{{\mathbb Q}}
\def\C{{\mathbb C}}
\def\l{{\langle}}
\def\r{\rangle}
\def\t{\tau}
\def\k{\kappa}
\def\a{\alpha}
\def\la{\lambda}
\def\De{\Delta}
\def\de{\delta}
\def\ga{\gamma}
\def\Ga{\Gamma}
\def\ep{\varepsilon}
\def\eps{\varepsilon}
\def\si{\sigma}
\def\Re {{\rm Re}\,}
\def\Im {{\rm Im}\,}
\def\E{{\mathbb E}}
\def\P{{\mathbb P}}
\def\Z{{\mathbb Z}}
\def\D{{\mathbb D}}
\newcommand{\ceil}[1]{\lceil{#1}\rceil}

%\allowdisplaybreaks
\title{Traveling waves of a full parabolic attraction-repulsion chemotaxis systems with logistic sources}

\author{
Rachidi B. Salako   \\
Department of Mathematics\\
The Ohio State University\\
Columbus, OH 43210,
U.S.A. }

\date{}
\maketitle
%\begin{document}

\begin{abstract}
In this paper, we study traveling wave solutions of the chemotaxis systems
\begin{equation}\label{1}
\begin{cases}
u_{t}=\Delta u -\chi_1\nabla( u\nabla v_1)+\chi_2 \nabla(u\nabla v_2 )+ u(a -b u), \qquad \  x\in\mathbb{R} \\
\tau\partial_tv_1=(\Delta- \lambda_1 I)v_1+ \mu_1 u,  \qquad \ x\in\mathbb{R},  \\
\tau\partial v_2=(\Delta- \lambda_2 I)v_2+ \mu_2 u,  \qquad \ \ x\in\mathbb{R},
\end{cases}
\end{equation}
where $\tau>0,\chi_{i}> 0,\lambda_i> 0,\ \mu_i>0$ ($i=1,2$) and  $\ a>0,\  b> 0$ are  constants, and $N$ is a positive integer. Under some appropriate conditions on the parameters, we show that there exist two positive constant $ 0<c^{*}(\tau,\chi_1,\mu_1,\lambda_1,\chi_2,\mu_2,\lambda_2)<c^{**}(\tau,\chi_1,\mu_1,\lambda_1,\chi_2,\mu_2,\lambda_2)$ such that  for every $c^{*}(\tau,\chi_1,\mu_1,\lambda_1,\chi_2,\mu_2,\lambda_2)\leq c<c^{**}(\tau,\chi_1,\mu_1,\lambda_1,\chi_2,\mu_2,\lambda_2)$, \eqref{1} has a traveling wave solution $(u,v_1,v_2)(x,t)=(U,V_1,V_2)(x-ct)$  connecting $(\frac{a}{b},\frac{a\mu_1}{b\lambda_1},\frac{a\mu_2}{b\lambda_2})$ and $(0,0,0)$ satisfying
$$
\lim_{z\to \infty}\frac{U(z)}{e^{-\mu z}}=1,
$$
where  $\mu\in (0,\sqrt a)$ is such that $c=c_\mu:=\mu+\frac{a}{\mu}$. Moreover,
$$
\lim_{(\chi_1,\chi_2)\to (0^+,0^+))}c^{**}(\tau,\chi_1,\mu_1,\lambda_1,\chi_2,\mu_2,\lambda_2)=\infty$$
and
$$\lim_{(\chi_1,\chi_2)\to (0^+,0^+))}c^{*}(\tau,\chi_1,\mu_1,\lambda_1,\chi_2,\mu_2,\lambda_2)= c_{\tilde{\mu}^*},
$$
 where  $\tilde{\mu}^*={\min\{\sqrt{a}, \sqrt{\frac{\lambda_1+\tau a}{(1-\tau)_{+}}},\sqrt{\frac{\lambda_2+\tau a}{(1-\tau)_{+}}}\}}$.  We also show  that  \eqref{1} has no traveling wave solution connecting $(\frac{a}{b},\frac{a\mu_1}{b\lambda_1},\frac{a\mu_2}{b\lambda_2})$ and $(0,0,0)$ with speed $c<2\sqrt{a}$. 
\end{abstract}



\section{Introduction and the Statement of the Main Results}
Chemotaxis describes the oriented movement of biological cells or organisms in response to chemical gradients. The oriented movement of cells has a crucial role in a wide range of biological phenomena. At the beginning of 1970s, Keller and Segel  (see \cite{KeSe1}, \cite{KeSe2}) introduced systems of
partial differential equations of the following form  to model the time evolution of both the density $u(x,t)$ of a mobile species and the density $v(x,t)$ of a  chemoattractant,
\begin{equation}\label{IntroEq0}
\begin{cases}
u_{t}=\nabla\cdot (m(u)\nabla u- \chi(u,v)\nabla v) + f(u,v),\quad   x\in\Omega \\
\tau v_t=\Delta v + g(u,v),\quad  x\in\Omega
\end{cases}
\end{equation}
complemented with certain boundary condition on $\partial\Omega$ if $\Omega$ is bounded, where $\Omega\subset \R^N$ is an open domain;  $\tau\ge 0$ is a non-negative constant linked to the speed of diffusion of the chemical;  the function $\chi(u,v)$ represents  the sensitivity with respect to chemotaxis; and the functions $f$ and $g$ model the growth of the mobile species and the chemoattractant, respectively.
In literature, \eqref{IntroEq0} is called the Keller-Segel model or a chemotaxis model.

Since the works by Keller and Segel,  a rich variety of mathematical models for studying chemotaxis has appeared (see \cite{BBTW, DiNa, DiNaRa, GaSaTe, HiPa1, HiPo, KKAS, NAGAI_SENBA_YOSHIDA, Sug, SuKu, TeWi,  win_jde, win_JMAA_veryweak, win_arxiv, Win, win_JNLS, YoYo, PCHu1, QZhanYLi, ZhMuHuTi}, and the references therein).
  The reader is referred to \cite{HiPa, Hor} for some detailed introduction into the mathematics of KS models. In the current paper,  we consider chemoattraction-repulsion process  on the whole space in which cells undergo random motion and chemotaxis towards attractant and away from repellent \cite{MLACLE, PCHu1, QZhanYLi}. Moreover, we consider the model with proliferation and death of cells. These lead to the model of partial differential equations as follows:
\begin{equation}\label{Main-eq01}
\begin{cases}
\partial_t u=\Delta u -\chi_1\nabla( u\nabla v_1)+\chi_2 \nabla(u\nabla v_2 )+ u(a -b u), \qquad \  x\in \R^N  \\
\tau\partial_t v_1=(\Delta- \lambda_1 I)v_1+ \mu_1 u,  \qquad \ x\in \R^N  \\
\tau \partial_t v_2=(\Delta- \lambda_2 I)v_2+ \mu_2 u,  \qquad \text{in}\ \ x\in \R^N.
\end{cases}
\end{equation}

The objective of the current paper is to study the existence of traveling wave solutions  of \eqref{Main-eq01} 
connecting $(\frac{a}{b},\frac{a\mu_1}{b\lambda_1},\frac{a\mu_2}{b\lambda_2})$ and $(0,0,0)$.
A nonnegative solution $(u(x,t),v_1(x,t),v_2(x,t))$ of \eqref{Main-eq01} defined for every $(x,t)\in \R^{N+1}$ is called a {\it traveling wave solution} connecting $(\frac{a}{b},\frac{a\mu_1}{b\lambda_1},\frac{a\mu_2}{b\lambda_2})$ and $(0,0,0)$ and propagating in the direction $\xi\in S^{N-1}$ with speed $c$ if it is of the form
$(u(x,t),v_1(x,t),v_2(x,t))=(U(x\cdot\xi-ct),V_1(x\cdot\xi-ct),V_2(x\cdot\xi-ct))$ with
$\lim_{z\to -\infty}(U(z),V_1(z),V_2(z))=(\frac{a}{b},\frac{a\mu_1}{b\lambda_1},\frac{a\mu_2}{b\lambda_2})$ and $\lim_{z\to\infty}(U(z),V_1(z),V_2)=(0,0,0)$.

Observe that, if $(u(x,t),v_1(x,t),v_2(x,t))=(U(x\cdot\xi-ct),V_1(x\cdot\xi-ct),V_2(x\cdot\xi-ct))$  $(x\in\R^N,  t\in\R)$ is a traveling wave solution of \eqref{Main-eq01} connecting  $(\frac{a}{b},\frac{a\mu_1}{b\lambda_1},\frac{a\mu_2}{b\lambda_2})$ and $(0,0,0)$ and propagating
in the direction $\xi\in S^{N-1}$, then $(u,v_1,v_2)=(U(x-ct),V_1(x-ct),V_2(x-ct))$ ($x\in\R$)
is a traveling wave solution of
\begin{equation}
\label{Main-eq1}
\begin{cases}
\partial_tu=\partial_{xx}u +\partial_x(u \partial_x(\chi_2v_2-\chi_1v_1)) + u(a-bu),\quad x\in\R,\cr
\tau\partial_tv_1=\partial_{xx}v_{1}-\lambda_1v_1+ \mu_1u, \quad x\in\R,\cr
\tau\partial_tv_2=\partial_{xx}v_{2}-\lambda_2v_2+ \mu_2u, \quad x\in\R,
\end{cases}
\end{equation}
connecting  $(\frac{a}{b},\frac{a\mu_1}{b\lambda_1},\frac{a\mu_2}{b\lambda_2})$ and $(0,0,0)$. Conversely, if $(u(x,t),v_1(x,t),v_2(x,t))=(U(x-ct),V_1(x-ct),V_2(x,t))$ ($x\in\R,  t\in\R$) is a traveling wave solution
of \eqref{Main-eq1} connecting $(\frac{a}{b},\frac{a\mu_1}{b\lambda_1},\frac{a\mu_2}{b\lambda_2})$ and $(0,0,0)$, then $(u,v_1,v_2)=(U(x\cdot\xi-ct),V_1(x\cdot\xi-ct),V_{2}(x\cdot\xi-ct))$  $(x\in\R^N)$ is a traveling wave solution of \eqref{Main-eq01} connecting  $(\frac{a}{b},\frac{a\mu_1}{b\lambda_1},\frac{a\mu_2}{b\lambda_2})$ and $(0,0,0)$ and propagating in the direction $\xi\in S^{N-1}$. In the following, we will then study the existence of traveling wave solutions
of \eqref{Main-eq1} connecting $(\frac{a}{b},\frac{a\mu_1}{b\lambda_1},\frac{a\mu_2}{b\lambda_2})$ and $(0,0,0)$.

Observe also that $(u,v_1,v_2)=(U(x-ct),V_1(x-ct),V_2(x-ct))$  is a traveling wave solution of \eqref{Main-eq1} connecting $(\frac{a}{b},\frac{a\mu_1}{b\lambda_1},\frac{a\mu_2}{b\lambda_2})$ and $(0,0,0)$ with speed $c$ if and only if $(u,v_1,v_2)=(U(x),V_1(x),V_2(x))$ is a stationary solution of
the following parabolic-elliptic-elliptic chemotaxis system,
\begin{equation}
\label{Main-eq2}
\begin{cases}
\partial_tu=\partial_{xx}u +c\partial_{x}u +\partial_x(u \partial_x(\chi_2v_2-\chi_1v_1)) + u(a-bu),\quad x\in\R,\cr
0=\partial_{xx}v_{1}+c\tau\partial_xv_1-\lambda_1v_1+\mu_1u, \quad x\in\R,\cr
0=\partial_{xx}v_{2}+c\tau\partial_xv_2-\lambda_2v_2+\mu_2u, \quad x\in\R,
\end{cases}
\end{equation}
 connecting  $(\frac{a}{b},\frac{a\mu_1}{b\lambda_1},\frac{a\mu_2}{b\lambda_2})$ and $(0,0,0)$.
In this paper, to study the existence of traveling
wave solutions of \eqref{Main-eq1}, we study the existence of constant $c$'s so that \eqref{Main-eq2} has a stationary solution $(U(x),V_1(x),V_2(x))$ satisfying
$(U(-\infty),V_1(-\infty),V_2(-\infty))=(\frac{a}{b},\frac{a\mu_1}{b\lambda_1},\frac{a\mu_2}{b\lambda_2})$ and $(U(\infty),V_1(\infty)$, $V_2(\infty))=(0,0,0)$.

To this end, we first establish some results on the global existence of classical solutions of \eqref{Main-eq2} and the stability of constant solution $(\frac{a}{b},\frac{a\mu_1}{b\lambda_1},\frac{a\mu_2}{b\lambda_2})$, which are of independent interest. Note that, for fixed $c$, it can be proved by the similar arguments as those in \cite{SaSh1} that for any $u_0\in C_{\rm unif}^b(\R)$ with $u_0\ge 0$, there is
$T_{\max}(u_0)\in (0,\infty]$ such that \eqref{Main-eq2} has a unique classical solution $(u(x,t;u_0),v_1(x,t;u_0),v_2(x,t;u_0))$ on $[0,T_{\max}(u_0))$ with
$u(x,0;u_0)=u_0(x)$. Furthermore, if $T_{\max}(u_0)<\infty$, then $\lim_{t\to T_{\max}^{-}(u_0)}\|u(\cdot,t;u_0)\|_{\infty}=\infty$.

In \cite{SaSh5} together with Wenxian Shen, we studied the existence of traveling wave solutions of \eqref{Main-eq1} when $\tau =0$. When $\tau>0$, the dynamics of \eqref{Main-eq1} is more complex and most of the  techniques developed in \cite{SaSh5}, for $\tau =0$, can not be adopted directly. So,  nontrivial modification and new  techniques are needed to handle the full parabolic system \eqref{Main-eq1}. Also, the results established in \cite{SaSh5}  make use of the stability of the positive constant equilibria proved in \cite{SaSh4}. To our best knowledge, the stability of positive  constant equilibria of \eqref{Main-eq1} still remains an open problem. In the current paper we established some new results  for $\tau=0$, mainly the existence of the so call "critical wave", see Theorem C (ii) below.

For clarity of the statements of our main results on the existence of global classical solutions  and stability of the steady solution $(\frac{a}{b},\frac{a\mu_1}{b\lambda_1},\frac{a\mu_2}{b\lambda_2})$ of \eqref{Main-eq2}, it would be convenience to introduce some definitions.  For every real number r, we let $(r)_+ = \max\{0,r\}$ and $(r)_{-} = \max\{0,-r\}$. Let
\begin{equation}\label{M-plus-equation}
\overline{M}:=\int_{0}^{\infty}\Big(\chi_2\lambda_2\mu_2 e^{-\lambda_2 s}-\chi_1\lambda_1\mu_1 e^{-\lambda_1 s}\Big)_{+}ds,
\end{equation}
\begin{equation}\label{M-minus-equation}
\underline{M}:=\int_{0}^{\infty}\Big(\chi_2\lambda_2\mu_2 e^{-\lambda_2 s}-\chi_1\lambda_1\mu_1 e^{-\lambda_1 s}\Big)_{-}ds,
\end{equation}
and
\begin{equation}\label{K-equation}
K:= \int_0^\infty\left|\frac{\mu_1\chi_1e^{-\lambda_1 s}}{2\sqrt{\pi s}}-\frac{\chi_2\mu_2e^{-\lambda_2 s}}{2\sqrt{\pi s}}\right|ds.
\end{equation}
Observe that
\begin{equation}\label{e10}
\overline{M}-\underline{M}=\int_{0}^{\infty}\Big(\chi_2\lambda_2\mu_2 e^{-\lambda_2 s}-\chi_1\lambda_1\mu_1 e^{-\lambda_1 s}\Big)ds=\chi_{2}\mu_2-\chi_1\mu_1.
\end{equation}

\medskip

Our main result on the existence of global solution of \eqref{Main-eq2} reads as follow.

\medskip

\noindent{\bf Theorem A. } {\it Suppose that $c\geq 0$ and  $b>\underline{M}+c\tau K$ where $\underline{M}$ and $K$ are given by \eqref{M-minus-equation} and \eqref{K-equation} respectively. Then for every nonnegative initial function $u_0\in C^{b}_{\rm unif}(\R)$, \eqref{Main-eq2} has a unique global classical solution $(u(x,t;u_0),v_1(x,t;u_0),v_2(x,t;u_0))$ satisfying $u(x,0;u_0)=u_0(x)$. Furthermore, it holds that
\begin{equation}\label{global-exixst-thm-eq1}
\|u(\cdot,t;u_0)\|_{\infty}\leq \max\left\{ \|u_0\|_{\infty}, \frac{a}{b-\underline{M}-c\tau K}\right\}, \quad \forall t\geq 0
\end{equation}
and
\begin{equation}\label{global-exixst-thm-eq2}
\|v_{i}(\cdot,t;u_0)\|_{\infty}\leq \max\left\{ \frac{\mu_i\|u_0\|_{\infty}}{\lambda_i}, \frac{a\mu_{i}}{(b-\underline{M}-c\tau K)\lambda_i}\right\}, \quad \forall t\geq 0, \ i=1,2.
\end{equation}
}

\medskip

\begin{rk} Note that  $\underline{M}\leq \chi_1\mu_1$ and equality holds if and only if $\chi_2=0$. Next,  we look at the case case $c=0$ in \eqref{Main-eq2}.  In this case, we have

1)  If $b\geq \chi_1\mu_1$, \eqref{Main-eq2} with $c=0$ and $\chi_2>0$ always has a unique bounded and nonnegative  global classical solution for every given $u_0\in C^{b}_{\rm uinf}(\R)$, $u_{0}\geq 0$.

2) If $\lambda_2>\lambda_1$ and $\chi_2\lambda_2\mu_2\leq \chi_1\lambda_1\mu_1$ then $\underline{M}=\chi_1\mu_1-\chi_2\mu_2$.

3) If  $\lambda_2>\lambda_1$ and $\chi_2\lambda_2\mu_2\geq \chi_1\lambda_1\mu_1$ then $ \underline{M}=\left(\frac{\lambda_2}{\lambda_1}-1\right)\chi_2\mu_2\left(\frac{\chi_1\lambda_1\mu_1}{\chi_2\lambda_2\mu_2}
\right)^{\frac{\lambda_2}{\lambda_2-\lambda_1}}.$

4) If  $\lambda_2<\lambda_1$ and $\chi_2\lambda_2\mu_2\leq \chi_1\lambda_1\mu_1$ then $\underline{M}= \chi_1\mu_1-\chi_2\mu_2+\left(\frac{\lambda_1}{\lambda_2}-1 \right)\chi_1\mu_1\left(\frac{\chi_2\lambda_2\mu_2}{\chi_1\lambda_1\mu_1}\right)^{\frac{\lambda_1}{\lambda_1-\lambda_2}} $. 

5) If  $\lambda_2<\lambda_1$ and $\chi_2\lambda_2\mu_2\geq  \chi_1\lambda_1\mu_1$ then $\underline{M}= 0.$

6) If $\lambda_1=\lambda_2$ then $\underline{M}=(\chi_2\mu_2-\chi_1\mu_1)_{-}$.

Therefore, in the case $c=0$, Theorem A improves Theorem A in \cite{SaSh4}.
\end{rk}

\medskip

Next, we state our result on the stability of the positive constant equilibrium solution of \eqref{Main-eq2}.

\medskip
\noindent{\bf Theorem B.} {\it Suppose that $c\geq 0$ and  $b>2(\underline{M}+c\tau K)$ where $\overline{M}$ and $K$ are given by \eqref{M-minus-equation} and \eqref{K-equation} respectively. Then for every initial function $u_0\in C^{b}_{\rm unif}(\R)$, with $\inf_{x\in\R}u_0(x)>0$,  the unique global classical solution $(u(x,t;u_0),v_1(x,t;u_0),v_2(x,t;u_0))$ satisfying $u(x,0;u_0)=u_0(x)$ of \eqref{Main-eq2} satisfies
\begin{equation}\label{asymp-thm-eq1}
\lim_{t\to\infty}\|u(\cdot,t;u_0)-\frac{a}{b}\|_{\infty}= 0
\end{equation}
and
\begin{equation}\label{global-exixst-thm-eq2}
\lim_{t\to\infty}\|v_{i}(\cdot,t;u_0)-\frac{a\mu_i}{b\lambda_i}\|_{\infty}= 0, \ i=1,2.
\end{equation}
}

 To  state  our main results on the existence of traveling wave solutions of \eqref{Main-eq1}  as well for our results in the subsequent sections, we introduce a few more notations. Let, for $0<\mu\leq \sqrt{a}$ and $\tau>0$,  $c_{\mu}=\mu+\frac{a}{\mu}$,
\begin{equation}
\overline{M}(s):=(\chi_2\mu_2\lambda_2e^{-\lambda_2 s}-\chi_1\mu_1\lambda_1e^{-\lambda_1 s})_{+}, \quad \forall\ s>0,
\end{equation}
\begin{equation}
\underline{M}(s):=(\chi_2\mu_2\lambda_2e^{-\lambda_2 s}-\chi_1\mu_1\lambda_1e^{-\lambda_1 s})_{-}, \quad \forall\ s>0,
\end{equation}
\begin{equation}\label{M-tau-plus-equation}
\overline{M}_{\tau,\mu}(s):=\overline{M}(s)e^{-(\tau\mu c_{\mu}-\mu^2)s}, \quad \forall\ \mu,s>0, \quad  \quad \overline{M}_{\tau,\mu}:=\int_0^{\infty}\overline{M}_{\tau,\mu}(s)ds,
\end{equation}
\begin{equation}\label{M-tau-minus-equation}
\underline{M}_{\tau,\mu}(s):=\underline{M}(s)e^{-(\tau\mu c_{\mu}-\mu^2)s}, \quad \forall\ \mu,s>0, \quad  \quad \underline{M}_{\tau,\mu}:=\int_0^{\infty}\underline{M}_{\tau,\mu}(s)ds,
\end{equation}
and
\begin{equation}\label{K-tau-equation}
K_{\tau,\mu}=\int_{0}^{\infty}\left|\chi_1\mu_1e^{-\lambda_1 s}-\chi_2\mu_2e^{-\lambda_2 s}\right|\frac{(1+\mu\sqrt{\pi s})e^{-(\tau\mu c_{\mu}-\mu^2)s}}{\sqrt{\pi s}}ds, \quad \forall \mu>0.
\end{equation}
Observe that  $\overline{M}=\int_{0}^{\infty}\overline{M}(s)ds$  and $\underline{M}=\int_{0}^{\infty}\underline{M}(s)ds$ where $\overline{M}$ and $\underline{M}$ are given by \eqref{M-plus-equation} and \eqref{M-minus-equation} respectively.

\medskip

Next, we state our results on the existence of traveling wave solutions of \eqref{Main-eq1}.  Let us consider the auxiliary function $f\ :\ (0, \min\{\sqrt{a}, \sqrt{\frac{\lambda_1+\tau a}{(1-\tau)_{+}}},\sqrt{\frac{\lambda_2+\tau a}{(1-\tau)_{+}}}\} ) \to\R$ defined by
$$ 
f(\mu)= \max\{2(\underline{M}+\tau c_{\mu}K)\ ,\ \chi_1\mu_1-\chi_2\mu_2+ (\tau c_{\mu}+\mu)K_{\tau,\mu}+\overline{M}_{\tau,\mu} \},
$$
where $c_{\mu}=\mu+\frac{a}{\mu}$, $\underline{M}$, $K$, $\overline{M}_{\tau,\mu}$, and $K_{\tau,\mu}$ are given by \eqref{M-minus-equation}, \eqref{K-equation}, \eqref{M-tau-plus-equation} and \eqref{K-tau-equation} respectively. Clearly, the function $f$ is continuous. We suppose that the following standing assumption holds.

\medskip

{\bf (H)} $b>\inf_{\mu}f(\mu)$. 
 
 \medskip
 
 Assuming that {\bf (H)} holds, we let  $(\mu^{**}_{\tau}, \mu^{*}_{\tau})$ denotes the right maximal open connected component of the open set $O=\{\mu\in(0, \min\{\sqrt{a}, \sqrt{\frac{\lambda_1+\tau a}{(1-\tau)_{+}}},\sqrt{\frac{\lambda_2+\tau a}{(1-\tau)_{+}}}\} )\ : \ b>f(\mu) \}$ and set 
 \begin{equation}\label{c-low-def} 
 c^{*}(\tau,\chi_1,\mu_1,\lambda_1,\chi_2,\mu_2,\lambda_2)=c_{\mu^*}
 \end{equation} 
 and
 \begin{equation}\label{c-up-def}
 c^{**}(\tau,\chi_1,\mu_1,\lambda_1,\chi_2,\mu_2,\lambda_2)=\lim_{\mu\to\mu^{**}+}c_{\mu}.
 \end{equation}
 
\medskip

\noindent{\bf Theorem C. } {\it Assume {\bf (H)}. Let $ c^{*}(\tau,\chi_1,\mu_1,\lambda_1,\chi_2,\mu_2,\lambda_2)$ and $c^{**}(\tau,\chi_1,\mu_1,\lambda_1,\chi_2,\mu_2,\lambda_2)$ be given \eqref{c-low-def} and \eqref{c-up-def} respectively. The following hold.
\begin{description}
\item[(i)]  For every $c^{*}(\tau,\chi_1,\mu_1,\lambda_1,\chi_2,\mu_2,\lambda_2)<c<c^{**}(\tau,\chi_1,\mu_1,\lambda_1,\chi_2,\mu_2,\lambda_2)$, \eqref{Main-eq1} has a traveling wave solution $(u,v_1,v_2)(x,t)=(U,V_1,V_2)(x-ct)$  connecting $(\frac{a}{b},\frac{a\mu_1}{b\lambda_1},\frac{a\mu_2}{b\lambda_2})$ and $(0,0,0)$ satisfying
\begin{equation}
\lim_{z\to \infty}\frac{U(z)}{e^{-\mu z}}=1,
\end{equation}
where  $\mu\in (0,\sqrt a)$ is such that $c=c_\mu:=\mu+\frac{a}{\mu}=c$.  Moreover,
$$
\lim_{(\chi_1,\chi_2)\to (0^+,0^+))}c^{**}(\tau,\chi_1,\mu_1,\lambda_1,\chi_2,\mu_2,\lambda_2)=\infty$$
and
$$\lim_{(\chi_1,\chi_2)\to (0^+,0^+))}c^{*}(\tau,\chi_1,\mu_1,\lambda_1,\chi_2,\mu_2,\lambda_2)=c_{\tilde{\mu}^*},
$$
where $\tilde{\mu}^*={\min\{\sqrt{a}, \sqrt{\frac{\lambda_1+\tau a}{(1-\tau)_{+}}},\sqrt{\frac{\lambda_2+\tau a}{(1-\tau)_{+}}}\}}$
\item[(ii)] There is a  traveling wave solution  $(u,v_1,v_2)(x,t)=(U,V_1,V_2)(x-ct)$ of \eqref{Main-eq1}  connecting $(\frac{a}{b},\frac{a\mu_1}{b\lambda_1},\frac{a\mu_2}{b\lambda_2})$ and $(0,0,0)$ with speed $c^{*}(\tau,\chi_1,\mu_1,\lambda_1,\chi_2,\mu_2,\lambda_2) $.
\item[(iii)] There is no traveling wave solution  $(u,v_1,v_2)(x,t)=(U,V_1,V_2))(x-ct)$ of \eqref{Main-eq1}  connecting $(\frac{a}{b},\frac{a\mu_1}{b\lambda_1},\frac{a\mu_2}{b\lambda_2})$ and $(0,0,0)$ with speed $c<2\sqrt{a}$.
\end{description}
}

\medskip

Let us make few comments about Theorem C. We first remark that $\tau =0$ is allowed in Theorem C. However, our result Theorem C (i) in this case is hard to compare with the results in \cite{SaSh5}. {The result in Theorem C (ii) is new  in the case of $\tau=0$.} Note also that when $\tau\ge 1$ and $b$ is sufficiently large, we have that $c^{*}(\tau,\chi_1,\mu_1,\lambda_1,\chi_2,\mu_2,\lambda_2) =2\sqrt{a}$, which is the minimal wave of the classical Fisher-KPP equation.  Hence in this case, it follows from Theorem C (ii) $\&$ (iii) that the presence of the chemotaxis signal does not affect the minimal wave speed. It remains open whether $c^{*}(\tau,\chi_1,\mu_1,\lambda_1,\chi_2,\mu_2,\lambda_2)$ can be taken to be $2\sqrt{a}$ in general. It also remains open whether $c^{**}(\tau,\chi_1,\mu_1,\lambda_1,\chi_2,\mu_2,\lambda_2)$ can be taken to  be  infinity. These questions are related
  to whether \eqref{Main-eq1} has a minimum and/or maximal traveling wave speeds. Note that in the absence of the chemotaxis effect, that is $\chi_1=\chi_2=0$, the first equation in \eqref{Main-eq1} becomes the following  scalar reaction diffusion equation,
\begin{equation}
\label{fisher-eq}
u_{t}=\Delta u + u(a-bu),\quad  x\in\R^N,\,\, t>0,
\end{equation}
which is referred to as Fisher or KPP equations due to  the pioneering works by Fisher (\cite{Fis}) and Kolmogorov, Petrowsky, Piscunov
(\cite{KPP}) on the spreading properties of \eqref{fisher-eq}.
It follows from the works \cite{Fis}, \cite{KPP}, and \cite{Wei1}  that $c^*_{-}$ and $c^*_{+}$  in Theorem C and Theorem D,  respectively, can be chosen so that $c^{\ast}_{-}=c^{\ast}_{+}=2\sqrt a$
 ($c^*:=2\sqrt a$ is called the {\it spatial spreading speed} of \eqref{fisher-eq} in literature), and that \eqref{fisher-eq} has traveling wave solutions $u(t,x)=\phi(x-ct)$  connecting $\frac{a}{b}$ and $0$ (i.e.
$(\phi(-\infty)=\frac{a}{b},\phi(\infty)=0)$) for all speeds $c\geq c^*$ and has no such traveling wave
solutions of slower speed.
 Since the pioneering works by  Fisher \cite{Fis} and Kolmogorov, Petrowsky,
Piscunov \cite{KPP},  a huge amount research has been carried out toward the spreading properties of
  reaction diffusion equations of the form,
\begin{equation}
\label{general-fisher-eq}
u_t=\Delta u+u f(t,x,u),\quad x\in\R^N,
\end{equation}
where $f(t,x,u)<0$ for $u\gg 1$,  $\p_u f(t,x,u)<0$ for $u\ge 0$ (see \cite{Berestycki1, BeHaNa1, BeHaNa2, Henri1, Fre, FrGa, LiZh, LiZh1, Nad, NoRuXi, NoXi1, She1, She2, Wei1, Wei2, Zla}, etc.). The existence of minimal wave speeds becomes a natural question to study as it is related to whether the presence of the chemotaxis speeds up or  slows down the minimal wave speed.  A partial and satisfactory answer to this is obtained  by Theorem C when $\tau\geq 1$ and the self-limitation rate  $b$  of the mobile species is sufficiently large. We plan to study these questions in our future works.

\medskip

The rest of the paper is organized as follows. In section 2 we study the dynamics of classical solutions of \eqref{Main-eq2} and prove Theorems A $\&$ B. Section 3  is  to develop the machinery and set up the right frame work to study the existence of traveling solution. Finally, based on the results established in section 3, we complete the proof of Theorem C in section 4.

\section{Dynamics of the  induced parabolic-elliptic-elliptic chemotaxis system}

In this section, we study the global existence of classical solutions of \eqref{Main-eq2} with given nonnegative initial functions
and the stability of the constant solution $(\frac{a}{b},\frac{a\mu_1}{b\lambda_1},\frac{a\mu_2}{b\lambda_2})$ of \eqref{Main-eq2}, and prove Theorems A and B.

%Consider
%\begin{equation}
%\label{reduced-eq2}
%\begin{cases}
%u_{t}=u_{xx}+c u_x -\chi (u v_x)_x + u(a-bu),\quad x\in\R\\
%0=v_{xx}+\tau cv_{x}-v+u, \quad x\in\R.
%\end{cases}
%\end{equation}
%Note that stationary solutions of \eqref{reduced-eq} and \eqref{reduced-eq2} are the same.
%To study traveling wave solutions of \eqref{main-eq} with speed $c$ is then equivalent to study the stationary solutions
%of \eqref{reduced-eq2}.

For fixed $c$, it can be proved by the similar arguments as those in \cite{SaSh1} that for any $u_0\in C_{\rm unif}^b(\R)$ with $u_0\ge 0$, there is
$T_{\max}(u_0)\in (0,\infty]$ such that \eqref{Main-eq2} has a unique classical solution $(u(x,t;u_0),v_1(x,t;u_0),v_1(x,t;u_0))$ on $[0,T_{\max}(u_0))$ with $u(x,0;u_0)=u_0(x)$.

 Observe that \eqref{Main-eq2} is equivalent to
\begin{equation}
\label{stationary-eq}
\begin{cases}
u_{t}=u_{xx}+(c +(\chi_2 v_2-\chi_1v_1)_{x})u_{x} + u(a+(\chi_2
\lambda_2v_2-\chi_1\lambda_1v_1)+c\tau( \chi_1v_1- \chi_2v_2)_{x})\\
\quad \quad-(b+\chi_2\mu_2-\chi_1\mu_1)u^2,\quad x\in\R\\
0=\partial_{xx}v_{1}+c\tau\partial_xv_{1}-\lambda_1v_1+\mu_1u, \quad x\in\R\\
0=\partial_{xx}v_{2}+c\tau\partial_xv_{2}-\lambda_2v_2+\mu_2u, \quad x\in\R.
\end{cases}
\end{equation}

\begin{proof}[Proof of Theorem A]
For given $u_0\in C_{\rm unif}^b(\R)$ and $T>0$, choose $C_0>0$ such that $0\le u_0\le C_0$ and $C_0\ge \frac{a}{b+\chi_2\mu_2-\chi_1\mu_1-\overline{M}-c\tau K}$,  and  let
$$
\mathcal{E}(u_0,T):=\{u\in C_{\rm unif}^b(\R\times [0,T])\,|\, u(x,0)=u_0(x),\,\, 0\le u(x,t)\le  C_0\}.
$$
For given $u\in \mathcal{E}(u_0,T)$, let $v_{1}(x,t;u)$ and $v_{2}(x,t;u)$ be the solutions of the second and third equations in \eqref{stationary-eq}.
Then for every $u\in\mathcal{E}(u_0,T)$,  $x\in\R, t\geq 0$, we have that
\begin{equation}\label{v-eq0001}
v_{i}(x,t;u)=\mu_i\int_{0}^{\infty}\int_{\R}\frac{e^{-\lambda_i s}}{\sqrt{4\pi s}}e^{-\frac{|x-z|^2}{4s}}u(z+\tau cs,t)dzds,
\end{equation}
and
\begin{equation}\label{partial-of-v-in-x-eq}
\partial_{x}v_{i}(x,t;u)=\frac{\mu_i}{\sqrt{\pi}}\int_0^\infty\int_\R\frac{(z-x)e^{-\lambda_i s}}{2s\sqrt{4\pi s}}e^{-\frac{|x-z|^2}{4s}}u(z+\tau cs,t)dzds.
\end{equation}
Using the fact that
$
\int_{\R} e^{-z^2}dz=\sqrt{\pi}$ and $\int_{0}^{\infty}ze^{-z^2}dz=\frac{1}{2}$, it follows from \eqref{v-eq0001} and \eqref{partial-of-v-in-x-eq} that

\begin{align*}
(\chi_2\lambda_2v_{2}-\chi_1\lambda_1v_{1})(x,t;u)&=\int_{0}^{\infty}\int_{\R}\Big(\chi_2\lambda_2\mu_2\frac{e^{-\lambda_2 s}}{\sqrt{\pi}}-\chi_1\lambda_1\mu_1\frac{e^{-\lambda_1 s}}{\sqrt{\pi}}\Big)e^{-z^2}u(x+2\sqrt{s}z+\tau cs,t)dzds\\
&\leq \Big[\int_{0}^{\infty}\int_{\R}\Big(\chi_2\lambda_2\mu_2\frac{e^{-\lambda_2 s}}{\sqrt{\pi}}-\chi_1\lambda_1\mu_1\frac{e^{-\lambda_1 s}}{\sqrt{\pi}}\Big)_{+}e^{-z^2}dzds\Big]\|u(\cdot,t)\|_{\infty}\\
&= \underbrace{\Big[\int_{0}^{\infty}\Big(\chi_2\lambda_2\mu_2 e^{-\lambda_2 s}-\chi_1\lambda_1\mu_1 e^{-\lambda_1 s}\Big)_{+}ds\Big]}_{\overline{M}}\|u(\cdot,t)\|_{\infty}
\end{align*}
and
\begin{align*}
\partial_{x}(\chi_1v_{1}-\chi_2v_{2})(x,t;u_0)&=\int_0^\infty\int_\R\left(\mu_1\chi_1e^{-\lambda_1 s}-\chi_2\mu_2e^{-\lambda_2 s}\right)\frac{(z-x)e^{-\frac{|x-z|^2}{4s}}}{2s\sqrt{4\pi s}}u(z+\tau cs,t)dzds\nonumber\\
&=\int_0^\infty\int_\R\left(\mu_1\chi_1e^{-\lambda_1 s}-\chi_2\mu_2e^{-\lambda_2 s}\right)\frac{ze^{-z^2}}{\sqrt{\pi s}}u(x+2\sqrt{s}z+\tau cs,t)dzds\nonumber\\
&\leq \Big[\int_0^\infty\int_{0}^{\infty}\left(\mu_1\chi_1e^{-\lambda_1 s}-\chi_2\mu_2e^{-\lambda_2 s}\right)_{+}\frac{ze^{-z^2}}{\sqrt{\pi s}}dzds\Big]\|u(\cdot,t)\|_{\infty}\nonumber\\
& -\Big[\int_0^\infty\int_{-\infty}^{0}\left(\mu_1\chi_1e^{-\lambda_1 s}-\chi_2\mu_2e^{-\lambda_2 s}\right)_{-}\frac{ze^{-z^2}}{\sqrt{\pi s}}dzds\Big]\|u(\cdot,t)\|_{\infty}\nonumber\\
&= \underbrace{\Big[\int_0^\infty\left|\frac{\mu_1\chi_1e^{-\lambda_1 s}}{2\sqrt{\pi s}}-\frac{\chi_2\mu_2e^{-\lambda_2 s}}{2\sqrt{\pi s}}\right|ds\Big]}_{K}\|u(\cdot,t)\|_{\infty}.
\end{align*}
Thus
\begin{equation}\label{Estimate-v-partial-v-eq0001}
(\chi_2\lambda_2v_{2}-\chi_1\lambda_1v_{1})(x,t;u)\leq \overline{M}C_0 \quad\text{and}\quad  \partial_{x}(\chi_1v_{1}-\chi_2v_{2})(x,t;u_0)\leq KC_0\quad \forall\, t\in [0,T],\,\, x\in\R.
\end{equation}

For given $u\in\mathcal{E}(u_0,T)$, let $\tilde U(x,t;u)$ be the solution of the initial value problem
\begin{equation}
\label{aux-eq1}
\begin{cases}\tilde U_t(x,t)= \tilde U_{xx}+(c+(\chi_2v_2-\chi_1v_1)_x(x,t;u))\tilde U_x+(a+(\chi_2\lambda_2v_2-\chi_1\lambda_1v_1)(x,t;u))\tilde{U}\cr
\qquad\qquad\,\,  +\tilde{U}\big(c\tau(\chi_1v_1-\chi_2v_2)_{x}(x,t;u)-(b+\chi_2\mu_2-\chi_1\mu_1) \tilde{U}\big),\quad x\in\R\cr
\tilde U(x,0;u)=u_0(x), \quad x\in\R.
\end{cases}
\end{equation}

Since $u_0\ge 0$, comparison principle for parabolic equations implies that $\tilde U(x,t)\ge 0$ for every $x\in\R, \ t\in[0, T]$. It follows from \eqref{Estimate-v-partial-v-eq0001} that
$$
\tilde U_t(x,t)\leq \tilde U_{xx}+(c+(\chi_2v_2-\chi_1v_1)_x(x,t;u))\tilde U_x+\tilde{U}\big(a+ (\overline{M} +c\tau K)C_0 -(b+\chi_2\mu_2-\chi_1\mu_1)\tilde U\big).
$$
%Let $\tilde U(x,t;u)$ be the solution of \eqref{aux-eq1} with $\tilde U(x,0;u)=u_0(x)$.
Hence, comparison principle for parabolic equations implies that
\begin{equation}\label{auxx-eq1}
\tilde U(x,t,u)\leq \tilde u(t,\|u_0\|_\infty), \quad \forall \ x\in\R,\ t\in[0, T],
\end{equation}
 where $\tilde u$ is the solution of the ODE
$$
\begin{cases}
\tilde u_t= \tilde u \big(a+ (\overline{M}+c\tau K)C_0-(b+\chi_2\mu_2-\chi_1\mu_1)\tilde u)\quad t>0\cr
\tilde u(0)=\|u_0\|_\infty.
\end{cases}
$$
Since $b+\chi_2\mu_2-\chi_1\mu_1>0$, the function $\tilde  u(\cdot,\|u_0\|_\infty)$ is defined for all time and satisfies $0\leq \tilde u(t,\|u_0\|_\infty)\leq \max\{\|u_0\|_{\infty},\ \frac{a+ (\overline{M}+c\tau K)C_0}{b+\chi_2\mu_2-\chi_1\mu_1}\}$ for every $t\geq 0$. This combined with \eqref{auxx-eq1} yields that
\begin{equation}\label{auxx-eq2}
0\le \tilde U(x,t;u)\le \max\{\|u_0\|,  \frac{a+ (\overline{M}+c\tau K)C_0}{b+\chi_2\mu_2-\chi_1\mu_1}\}\leq C_0\quad \forall\,\, x\in\R,\,\,\forall\,\, t\in [0,T].
\end{equation}
Note that the second inequality in \eqref{auxx-eq2} follows from the fact $\frac{a+ (\overline{M}+c\tau K)C_0}{b+\chi_2\mu_2-\chi_1\mu_1}\leq C_0$ whenever $C_0\ge \frac{a}{b+\chi_2\mu_2-\chi_1\mu_1-\overline{M}-c\tau K}$.  It then follows that  $\tilde U(\cdot,\cdot;u)\in\mathcal{E}(u_0,T)$.

\smallskip

Following the proof of Lemma 4.3 in \cite{SaSh2}, we can prove that the mapping $\mathcal{E}(u_0,T)\ni u\mapsto \tilde U(\cdot,\cdot;u)\in\mathcal{E}(u_0,T)$ has a fixed point
$\tilde U(x,t;u)=u(x,t)$. Note that if $\tilde U(\cdot,\cdot,u)=u$, then $(u(\cdot,\cdot),v_{1}(\cdot,\cdot;u),v_{2}(\cdot,\cdot;u))$ is a solution of \eqref{Main-eq2}. Since $u(\cdot,\cdot;u_0)$ is the only solution of \eqref{Main-eq2}, thus $u(\cdot,\cdot,u_0)=u(\cdot,\cdot)$. Hence, it follows from \eqref{auxx-eq2} that for any $T>0$,
\begin{equation}\label{u-eq0001}
0\le u(x,t;u_0)\le \max\{\|u_0\|,  \frac{a+ (\overline{M}+c\tau K)C_0}{b+\chi_2\mu_2-\chi_1\mu_1}\}\quad \forall\,\, t\in [0,T].
\end{equation}
 This implies that $T_{\max}(u_0)=\infty$. Recall fron \eqref{e10} that $\chi_2\mu_2-\chi_1\mu_1-\overline{M}=-\underline{M}$. Thus,  inequalities \eqref{global-exixst-thm-eq1} and \eqref{global-exixst-thm-eq2} follow from \eqref{u-eq0001} with $C_{0}=\max\{\|u_0\|_{\infty},\frac{a}{b+\chi_2\mu_2-\chi_1\mu_1-\overline{M}-c\tau K}\}$.
\end{proof}

In the next result, we prove the stability of the positive constant solution $(\frac{a}{b},\frac{a\mu_1}{b\lambda_1},\frac{a\mu_2}{b\lambda_2})$ of \eqref{Main-eq2}.

\begin{proof}[Proof of Theorem B]
Let $u_{0}\in C^{b}_{\rm uinf}(\R)$ with $\inf_{x\in\R}u_0(x)>0$ be given. By Theorem A, we have that
 $ \sup_{x\in\R,\ t\geq 0}u(x,t;u_0)<\infty$. Hence
$$
m:=\sup_{x\in\R,\ t\ge 0}(|\chi_1\lambda_1v_1-\chi_2\lambda_2v_2|(x,t,u_0)+c\tau |(\chi_1v_1-\chi_2v_2)_{x}|(x,t;u_0)|)<\infty.
$$
Thus, we have that
$$
u_{t}\geq u_{xx}+(c+(\chi_2v_2-\chi_1 v_1)_{x})u_x +(a-m-(b+\chi_2\mu_2-\chi_1\mu_1)u), \quad \forall x\in\R,\ t>0.
$$
Therefore, comparison principle for parabolic equations implies that
$$
u(x,t;u_{0})\geq U^{l}(t),\quad \forall\ x\in\R, \forall\ t\geq 0,
$$
where $U^{l}(t)$ is the solution of the ODE
$$
\begin{cases}
U_{t}=U(a-m-(b+\chi_2\mu_2-\chi_1\mu_1)U)\cr
U(0)=\inf_{x}u_0(x).
\end{cases}
$$
Since $\inf_{x}u_0(x)>0$, we have that $U^{l}(t)>0$ for every $t\geq 0$. Thus
\begin{equation}\label{e00}
0<U^{l}(t)\leq \inf_{x\in\R}u(x,t;u_0),\quad \forall\,\,  t\geq 0.
\end{equation}

 Next, let us define
$$
\bar u=\limsup_{t\to\infty}\sup_{x\in\R} u(x,t;u_0),\quad \underline u=\liminf_{t\to\infty}\inf_{x\in\R} u(x,t;u_0).
$$
Then for any $\epsilon>0$, there is $T_\epsilon>0$ such that
\begin{equation}\label{e000}
\underline u-\epsilon<u(x,t;u_0)\le \bar u+\epsilon\quad \forall t\ge T_\epsilon,\quad x\in\R.
\end{equation}

It follows from \eqref{v-eq0001} and \eqref{e000} that for every $x\in\R$ and $t\geq T_{\varepsilon}$ we have that
\begin{align}\label{e0}
(\chi_2\lambda_2v_2-\chi_1\lambda_1v_1)(x,t;u_0)&=\int_{0}^{\infty}\int_{\R}\Big(\chi_2\lambda_2\mu_2\frac{e^{-\lambda_2 s}}{\sqrt{\pi}}-\chi_1\lambda_1\mu_1\frac{e^{-\lambda_1 s}}{\sqrt{\pi}}\Big)e^{-z^2}u(x+2\sqrt{s}z+\tau cs,t)dzds\nonumber\\
&\leq (\overline{u}+\varepsilon)\int_{0}^{\infty}\int_{\R}\Big(\chi_2\lambda_2\mu_2\frac{e^{-\lambda_2 s}}{\sqrt{\pi}}-\chi_1\lambda_1\mu_1\frac{e^{-\lambda_1 s}}{\sqrt{\pi}}\Big)_{+}e^{-z^2}dzds\nonumber\\
&-(\underline{u}-\varepsilon)\int_{0}^{\infty}\int_{\R}\Big(\chi_2\lambda_2\mu_2\frac{e^{-\lambda_2 s}}{\sqrt{\pi}}-\chi_1\lambda_1\mu_1\frac{e^{-\lambda_1 s}}{\sqrt{\pi}}\Big)_{-}e^{-z^2}dzds\nonumber\\
&= (\overline{u}+\varepsilon)\int_{0}^{\infty}\Big(\chi_2\lambda_2\mu_2e^{-\lambda_2 s}-\chi_1\lambda_1\mu_1e^{-\lambda_1 s}\Big)_{+}ds\nonumber\\
&-(\underline{u}-\varepsilon)\int_{0}^{\infty}\Big(\chi_2\lambda_2\mu_2e^{-\lambda_2 s}-\chi_1\lambda_1\mu_1e^{-\lambda_1 s}\Big)_{-}ds\nonumber\\
&= (\overline{u}+\varepsilon)\overline{M}-(\underline{u}-\varepsilon)\underline{M}.
\end{align}
Similar arguments leading to the last inequality yield that
\begin{equation}\label{e1}
(\chi_2\lambda_2v_2-\chi_1\lambda_1v_1)(x,t;u_0)\geq (\underline{u}-\varepsilon)\overline{M}-(\overline{u}+\varepsilon)\underline{M}.
\end{equation}
Similarly, using \eqref{partial-of-v-in-x-eq} and \eqref{e000}, we have for every $x\in\R$ and $t\geq T_{\varepsilon}$,
\begin{align}\label{e2}
\partial_{x}(\chi_1v_{1}-\chi_2v_{2})(x,t;u_0)&=\int_0^\infty\int_\R\left(\mu_1\chi_1e^{-\lambda_1 s}-\chi_2\mu_2e^{-\lambda_2 s}\right)\frac{ze^{-z^2}}{\sqrt{\pi s}}u(x+2\sqrt{s}z+\tau cs,t)dzds\nonumber\\
&\leq (\overline{u}+\varepsilon)\int_0^\infty\int_{0}^{\infty}\left(\mu_1\chi_1e^{-\lambda_1 s}-\chi_2\mu_2e^{-\lambda_2 s}\right)_{+}\frac{ze^{-z^2}}{\sqrt{\pi s}}dzds\nonumber\\
&- (\underline{u}-\varepsilon)\int_0^\infty\int_{0}^{\infty}\left(\mu_1\chi_1e^{-\lambda_1 s}-\chi_2\mu_2e^{-\lambda_2 s}\right)_{-}\frac{ze^{-z^2}}{\sqrt{\pi s}}dzds\nonumber\\
&+(\underline{u}-\varepsilon)\int_0^\infty\int_{-\infty}^{0}\left(\mu_1\chi_1e^{-\lambda_1 s}-\chi_2\mu_2e^{-\lambda_2 s}\right)_{+}\frac{ze^{-z^2}}{\sqrt{\pi s}}dzds\nonumber\\
&- (\overline{u}+\varepsilon)\int_0^\infty\int_{-\infty}^{0}\left(\mu_1\chi_1e^{-\lambda_1 s}-\chi_2\mu_2e^{-\lambda_2 s}\right)_{-}\frac{ze^{-z^2}}{\sqrt{\pi s}}dzds\nonumber\\
&=(\overline{u}+\varepsilon)\int_{0}^{\infty}\left|\frac{\mu_1\chi_1e^{-\lambda_1s}}{2\sqrt{s}} -\frac{\mu_2\chi_2e^{-\lambda_2s}}{2\sqrt{s}}\right|ds\nonumber\\
& -(\underline{u}-\varepsilon)\int_{0}^{\infty}\left|\frac{\mu_1\chi_1e^{-\lambda_1s}}{2\sqrt{s}} -\frac{\mu_2\chi_2e^{-\lambda_2s}}{2\sqrt{s}}\right|ds\nonumber\\
&=(\overline{u}-\underline{u}+2\varepsilon)K,
\end{align}
 and
 \begin{equation}\label{e3}
 \partial_{x}(\chi_1v_{1}-\chi_2v_{2})(x,t;u_0)\geq (\underline{u}-\varepsilon)K -(\overline{u}+\varepsilon)K=(\underline{u}-\overline{u}-2\varepsilon)K.
 \end{equation}

Thus, using inequalities \eqref{e0} and \eqref{e2}, for every $t\geq T_{\varepsilon}$, we have that
$$
u_t\leq u_{xx}+(c+(\chi_2v_2-\chi_1v_{1})_{x})u_x+(a+(\overline{u}+\varepsilon)(c\tau K+\overline{M})-(\underline{u}-\varepsilon)(c\tau K+\underline{M})-(b+\chi_2\mu_2-\chi_1\mu_1)u ).
$$
Thus, comparison principle for parabolic equations implies that
\begin{equation}\label{e5}
u(x,t;u_0)\leq \overline{U}^{\varepsilon}(t),\quad \forall\, x\in\R,\ \forall\, t\geq T_{\varepsilon},
\end{equation}
where $\overline{U}^{\varepsilon}(t)$ is the solution of the ODE
$$
\begin{cases}
\overline{U}_{t}=\overline{U}(a+(\overline{u}+\varepsilon)(c\tau K+\overline{M})-(\underline{u}-\varepsilon)(c\tau K+\underline{M})-(b+\chi_2\mu_2-\chi_1\mu_1)\overline{U} )\quad t> T_{\varepsilon}\cr
\overline{U}(T_{\varepsilon})=\|u(\cdot,T_{\varepsilon};u_0)\|_{\infty}.
\end{cases}
$$
Since $\|u(\cdot,T_{\varepsilon};u_0)\|_{\infty}>0$ and $b+\chi_2\mu_2-\chi_1\mu_1>0$, we have that $\overline{U}^{\varepsilon}(t)$ is defined for all time and satisfies
$$ \lim_{t\to\infty}\overline{U}^{\varepsilon}(t)= \frac{(a+(\overline{u}+\varepsilon)(c\tau K+\overline{M})-(\underline{u}-\varepsilon)(c\tau K+\underline{M}) )_{+}}{b+\chi_2\mu_2-\chi_1\mu_1}.$$
Thus, it follows from \eqref{e5} that
$$
\overline{u}\leq \frac{(a+(\overline{u}+\varepsilon)(c\tau K+\overline{M})-(\underline{u}-\varepsilon)(c\tau K+\underline{M}) )_{+}}{b+\chi_2\mu_2-\chi_1\mu_1},\quad \forall\, \varepsilon>0.
$$
Letting $\varepsilon\to0$ in the last inequality, we obtain
\begin{equation*}
\overline{u}\leq \frac{(a+\overline{u}(c\tau K+\overline{M})-\underline{u}(c\tau K+\underline{M}) )_{+}}{b+\chi_2\mu_2-\chi_1\mu_1}.
\end{equation*}
But, note that $ (a+\overline{u}(c\tau K+\overline{M})-\underline{u}(c\tau K+\underline{M}) )_{+}=0$ would implies that $\underline{u}=\overline{u}=0$. Which in turn yields that
$$
0=(a+\overline{u}(c\tau K+\overline{M})-\underline{u}(c\tau K+\underline{M}) )_{+}=a.
$$
This is impossible since $a>0$. Hence $(a-\chi\underline{u}+\frac{\chi\tau c}{2}(\overline{u}-\underline{u}))_{+}>0$. Whence
\begin{equation}\label{e6}
\overline{u}\leq \frac{a+\overline{u}(c\tau K+\overline{M})-\underline{u}(c\tau K+\underline{M})}{b+\chi_2\mu_2-\chi_1\mu_1}
\end{equation}

Next, using again inequalities \eqref{e1} and \eqref{e3}, for every $t\geq T_{\varepsilon}$, we have that
$$
u_t\geq u_{xx}+(c+(\chi_2v_2-\chi_1 v_1)_{x})u_x+(a+(\underline{u}-\varepsilon)(c\tau K +\overline{M})-(\overline{u}+\varepsilon)(c\tau K+\underline{M})-(b+\chi_2\mu_2-\chi_1\mu_1)u ).
$$
Thus, comparison principle for parabolic equations implies that
\begin{equation}\label{e7}
u(x,t;u_0)\geq \underline{U}^{\varepsilon}(t),\quad \forall\, x\in\R,\ \forall\, t\geq T_{\varepsilon},
\end{equation}
where $\underline{U}^{\varepsilon}(t)$ is the solution of the ODE
$$
\begin{cases}
\underline{U}_{t}^\varepsilon=\underline{U}^\varepsilon(a+(\underline{u}-\varepsilon)(c\tau K +\overline{M})-(\overline{u}+\varepsilon)(c\tau K+\underline{M})-(b+\chi_2\mu_2-\chi_1\mu_1)\underline{U}^\varepsilon)\quad t> T_{\varepsilon}\cr
\underline{U}^\varepsilon(T_{\varepsilon})=\inf_{x\in\R}u(x,T_{\varepsilon},u_0).
\end{cases}
$$
From \eqref{e00} we know that  $\inf_{x\in\R}u(x,T_{\varepsilon},u_0)>0$.  Thus, using the fact $b+\chi_2\mu_2-\chi_1\mu_1>0$, we have that $\underline{U}^\varepsilon$ is defined for all time and satisfies
$$ \lim_{t\to\infty}\underline{U}^\varepsilon= \frac{(a+(\underline{u}-\varepsilon)(c\tau K +\overline{M})-(\overline{u}+\varepsilon)(c\tau K+\underline{M}))_{+}}{b+\chi_2\mu_2-\chi_1\mu_1}.$$ Thus, it follows from \eqref{e7} that
$$
\underline{u}\geq \frac{(a+(\underline{u}-\varepsilon)(c\tau K +\overline{M})-(\overline{u}+\varepsilon)(c\tau K+\underline{M}))_{+}}{b+\chi_2\mu_2-\chi_1\mu_1},\quad \forall\, \varepsilon>0.
$$
Letting $\varepsilon$ tends to 0 in the last inequality, we obtain that
\begin{equation}\label{e8}
\underline{u}\geq \frac{(a+\underline{u}(c\tau K +\overline{M})-\overline{u}(c\tau K+\underline{M}))_{+}}{b+\chi_2\mu_2-\chi_1\mu_1}.
\end{equation}
It follows from inequality \eqref{e6} and \eqref{e8} that
$$
(b+\chi_2\mu_2-\chi_1\mu_1)(\overline{u}-\underline{u})\leq  (2c\tau K+\overline{M}+\underline{M})(\overline{u}-\underline{u}).
$$
Which is equivalent to
\begin{equation}\label{e9}
(b+\chi_2\mu_2-\chi_1\mu_1-2c\tau K-\overline{M}-\underline{M})(\overline{u}-\underline{u})\leq 0.
\end{equation}
Observe that  $b+\chi_2\mu_2-\chi_1\mu_1-2c\tau K-\overline{M}-\underline{M}=b+2\chi_2\mu_2-2\chi_1\mu_1-2c\tau K>0 $. Thus, it follows from \eqref{e9} that $\overline{u}=\underline{u}$. Thus it follows from \eqref{e6}  and \eqref{e8} that $(b+\chi_2\mu_2-\chi_1\mu_1-\overline{M}+\underline{M})\overline{u}= a $. Combining this with \eqref{e10}, we conclude that $\overline{u}=\underline{u}=\frac{a}{b}$.
\end{proof}

\section{Super- and Sub- solutions}

In this section, we will construct super- and sub-solutions of some related equations of \eqref{Main-eq2}, which will be used to prove the existence of traveling wave solutions of  \eqref{Main-eq1} in next section. Throughout this section we suppose that $a>0$  and $b>0$ are given positive real numbers.

 Note that, for given $c$, to show the existence of a traveling wave solution of  \eqref{Main-eq1} connecting $(\frac{a}{b},\frac{a\mu_1}{b\lambda_1},\frac{a\mu_2}{b\lambda_2})$ and $(0,0,0)$ is  equivalent to show the existence of a stationary solution connecting $(\frac{a}{b},\frac{a\mu_1}{b\lambda_1},\frac{a\mu_2}{b\lambda_2})$ and $(0,0,0)$.

For every $\tau>0,\ 0<\mu <\min\{\sqrt{a}, \sqrt{\frac{\lambda_1+\tau a}{(1-\tau)_{+}}},\sqrt{\frac{\lambda_2+\tau a}{(1-\tau)_{+}}}\}$ and $x\in \R$ define
$$\varphi_{\tau,\mu}(x)=e^{-\mu x}$$ and set
$$
c_{\mu}=\mu+\frac{a}{\mu}.
 $$

Note that for every fixed  $\tau>0$ and $0<\mu  < \min\{\sqrt{a}, \sqrt{\frac{\lambda_1+\tau a}{(1-\tau)_{+}}},\sqrt{\frac{\lambda_2+\tau a}{(1-\tau)_{+}}}\}$,  $\lambda_i+\tau \mu c_\mu-\mu^2>0$ and the function $\varphi_{\tau,\mu}$ is decreasing, infinitely many differentiable, and satisfies
 \begin{equation}\label{Eq1 of varphi}
\varphi_{\tau,\mu}''(x)+c_{\mu}\varphi_{\tau,\mu}'(x)+a\varphi_{\tau,\mu}(x)=0 \quad\forall\ x\in\R
\end{equation}
and
\begin{equation}\label{Eq2 of varphi}
 \varphi_{\tau,\mu}''(x)+\tau  c_\mu\varphi' _{\tau,\mu}(x)-\lambda_i\varphi_{\tau,\mu}(x)+(\lambda_i+\tau \mu c_\mu -\mu^2)\varphi_{\tau,\mu}(x)=0\quad \forall\,\, x\in\R.
\end{equation}

For every $C_0>0,\ \tau>0$ and $0<\mu<\min\{\sqrt{a}, \sqrt{\frac{\lambda_1+\tau a}{(1-\tau)_{+}}}, \sqrt{\frac{\lambda_2+\tau a}{(1-\tau)_{+}}}\}$   define
\begin{equation}
U_{\tau,\mu,C_0}^{+}(x)=\min\{C_0, \varphi_{\tau,\mu}(x)\}=\begin{cases}
C_0 \ \quad \text{if }\ x\leq \frac{-\ln(C_0)}{\mu}\\
e^{-\mu x} \quad \ \text{if}\ x\geq \frac{-\ln(C_0)}{\mu}.
\end{cases}
\end{equation}
Since $\varphi_{\tau,\mu}$ is ia non decreasing, then the functions $U^{+}_{\tau,\mu, C_0}$ is non-increasing. Furthermore, the functions $U^{+}_{\tau,\mu,C_0}$ belongs to $C^{\delta}_{\rm unif}(\R)$ for every $0\leq \delta< 1$, $\tau>0$,  $0< \mu< \min\{\sqrt{a}, \sqrt{\frac{\lambda_1+\tau a}{(1-\tau)_{+}}}, \sqrt{\frac{\lambda_2+\tau a}{(1-\tau)_{+}}}\}$, and $C_0>0$.

Let $C_0>0,\ \tau>0$ and $0< \mu<\min\{\sqrt{a}, \sqrt{\frac{\lambda_1+\tau a}{(1-\tau)_{+}}},\sqrt{\frac{\lambda_2+\tau a}{(1-\tau)_{+}}}\}$ be fixed. Next, let $\mu<\tilde{\mu}<\min\{2\mu, \sqrt{a}, \sqrt{\frac{\lambda_1+\tau a}{(1-\tau)_{+}}},\sqrt{\frac{\lambda_2+\tau a}{(1-\tau)_{+}}}\}$ and $d>\max\{1, C_0^{\frac{\mu-\tilde{\mu}}{\mu}}\}$. The function $\varphi_{\tau,\mu}-d\varphi_{\tau,\tilde{\mu}}$ attains its maximum value at $\bar{a}_{\mu,\tilde{\mu},d}:=\frac{\ln(d\tilde{\mu})-\ln(\mu)}{\tilde{\mu}-\mu}$ and takes the value zero at $\underline{a}_{\mu,\tilde{\mu},d}:= \frac{\ln(d)}{\tilde{\mu}-\mu}$.
%For every  $\underline{a}_{\mu,\tilde{\mu},d} < x_{1} < \bar{a}_{\mu,\tilde{\mu},d}$, we shall associate the open interval $\Omega_{x_{1}}=(x_{1}\ ,  \ x_{2})$ where $x_{2}>\bar{a}_{\mu,\tilde{\mu},d}$ is the only real number satisfying
%\begin{equation}
%\varphi_{\mu}(x_{1})-d\varphi_{\tilde{\mu}}(x_{1})=\varphi_{\mu}(x_{2})-d\varphi_{\tilde{\mu}}(x_{2}).
%\end{equation}
Define
\begin{equation}
U_{\tau,\mu,C_0}^{-}(x):= \max\{ 0, \varphi_{\tau,\mu}(x)-d\varphi_{\tau,\tilde{\mu}}(x)\}=\begin{cases}
0\qquad \qquad \qquad \quad \text{if}\ \ x\leq \underline{a}_{\mu,\tilde{\mu},d}\\
\varphi_{\tau,\mu}(x)-d\varphi_{\tau,\tilde{\mu}}(x)\quad \text{if}\ x\geq \underline{a}_{\mu,\tilde{\mu},d}.
\end{cases}
\end{equation}
From the choice of $d$, it follows that  $0\leq U_{\tau,\mu,C_0}^{-}\leq U^{+}_{\tau,\mu,C_0}\leq C_0$ and $U_{\tau,\mu,C_0}^{-}\in C^{\delta}_{\rm unif}(\R)$ for every $0\leq \delta< 1$. Finally, let us consider the set $\mathcal{E}_{\tau,\mu}(C_0)$ defined by
\begin{equation}\label{definition-E-mu}
\mathcal{E}_{\tau,\mu}(C_0)=\{u\in C^{b}_{\rm unif}(\R) \,|\, U_{\tau,\mu,C_0}^{-}\leq u\leq U_{\tau,\mu,C_0}^{+}\}.
\end{equation}
It should be noted that $U_{\tau,\mu,C_0}^{-}$ and $\mathcal{E}_{\tau,\mu}(C_0)$ all depend on $\tilde{\mu}$ and $d$. Later on, we shall provide more information on how to choose $d$ and $\tilde{\mu}$ whenever $\tau$, $\mu$  and $C_0$ are given.

For every $u\in C_{\rm unif}^b(\R)$, consider
\begin{align}\label{ODE2}
U_{t}=&U_{xx}+(c_{\mu}+(\chi_2V_2-\chi_1V_1)_{x}(x;u))U_{x}+(a+(\chi_2\lambda_2V_2-\chi_1\lambda_1V_1)(x;u))U\nonumber\\
&+(c\tau(\chi_1V_1-\chi_2V_2)_{x}(x,u)-(b+\chi_2\mu_2-\chi_1\mu_1)U)U, \quad x\in \R,
\end{align}
where
\begin{equation}\label{Inverse of u}
V_i(x;u)=\mu_i\int_{0}^{\infty}\int_{\R}\frac{e^{-\lambda_is}}{\sqrt{4\pi s}}e^{-\frac{|x-z|^{2}}{4s}}u(z+\tau c_{\mu}s)dzds.
\end{equation}
For given $u\in C_{\rm unif}^b(\R)$, it is well known that the function $V_{1}(x;u)$
 and $V_{2}(x,u)$ are the solutions of the second and third equations of \eqref{Main-eq2} in $C^{b}_{\rm unif}(\R)$ respectively.

For  given open intervals $D\subset \R$ and $I\subset \R$, a function $U(\cdot,\cdot)\in C^{2,1}(D\times I,\R)$ is called a {\it super-solution} (respectively {\it sub-solution}) of \eqref{ODE2} on $D\times I$ if
\begin{align*}
U_{t}\ge & U_{xx}+(c_{\mu}+(\chi_2V_2-\chi_1V_1)_{x}(x;u))U_{x}+(a+(\chi_2\lambda_2V_2-\chi_1\lambda_1V_1)(x;u))U\nonumber\\
&+(c_{\mu}\tau(\chi_1V_1-\chi_2V_2)_{x}(x,u)-(b+\chi_2\mu_2-\chi_1\mu_1)U)U, \quad {\rm for}\,\, x\in D,\,\,\, t\in I
\end{align*}
(respectively
\begin{align*}
U_{t} \leq & U_{xx}+(c_{\mu}+(\chi_2V_2-\chi_1V_1)_{x}(x;u))U_{x}+(a+(\chi_2\lambda_2V_2-\chi_1\lambda_1V_1)(x;u))U\nonumber\\
&+(c_{\mu}\tau(\chi_1V_1-\chi_2V_2)_{x}(x,u)-(b+\chi_2\mu_2-\chi_1\mu_1)U)U, \quad {\rm for}\,\, x\in D,\,\,\, t\in I.
\end{align*})

Next, we state the main result of this section. For convenience, we introduce the following standing assumption.

\noindent { {\bf (H)} {\it $\tau>0$, $0<\mu<\min\{\sqrt{a}, \sqrt{\frac{\lambda_1+\tau a}{(1-\tau)_{+}}},\sqrt{\frac{\lambda_2+\tau a}{(1-\tau)_{+}}}\}$,  $b+\chi_2\mu_2-\chi_1\mu_1>(\tau c_{\mu}+\mu)K_{\tau,\mu}+\overline{M}_{\tau,\mu}$, and
\begin{align}\label{Eq01_Th1}
b+2(\chi_2\mu_2-\chi_1\mu_1)>2(\overline{M}+\tau c_{\mu}K).
\end{align}
where $\overline{M}$, $K$, $\overline{M}_{\tau,\mu}$,  and $K_{\tau,\mu}$ are given by \eqref{M-plus-equation}, \eqref{K-equation}, \eqref{M-tau-plus-equation},  and \eqref{K-tau-equation} respectively.}
}

\begin{tm}\label{super-sub-solu-thm}
Assume {\bf (H)}. Then  the following hold.

\begin{itemize}

\item[(1)]There is a positive real number $\tilde{C}_{0}>0$, $\tilde{C}_0=\tilde{C}_{0}(\tau,\chi_1,\lambda_1,\mu_1,\chi_2,\lambda_2,\mu_2,\mu)$, such for every $C_{0}\geq \tilde{C}_{0}$, and for every $u\in \mathcal{E}_{\tau,\mu}(C_0)$, we have that  $U(x,t)=C_0$ is  supper-solutions of \eqref{ODE2} on $\R\times\R$.

\item[(2)] For every $C_0>0$ and for every $u\in \mathcal{E}_{\tau,\mu}(C_0)$, $U(x,t)=\varphi_{\tau,\mu}(x)$ is a supper-solutions of \eqref{ODE2} on $\R\times\R$.

\item[(3)] For every $C_0>0$, there is $d_0>\max\{1,C_{0}^{\frac{\mu-\tilde{\mu}}{\mu}}\}$, $d_0=d_{0}(\tau,\chi_1,\lambda_1,\mu_1,\chi_2,\lambda_2,\mu_2,\mu)$, such  that for every $u\in \mathcal{E}_{\tau,\mu}(C_0)$, we have that  $U(x,t)=U_{\tau,\mu,C_0}^-(x)$ is a sub-solution of \eqref{ODE2} on
$(\underline{a}_{\mu,\tilde{\mu},d},\infty)\times \R$ for all $d\ge d_0$ and $\mu< \tilde{\mu}<\min\{\sqrt{a},\ \sqrt{\frac{\lambda_1+\tau a}{(1-\tau)_{+}}},\ \sqrt{\frac{\lambda_2+\tau a}{(1-\tau)_{+}}},\ 2\mu,\mu+\frac{b+2\chi_2\mu_2-\chi_1\mu_1-(\tau c_{\mu}+\mu)K_{\tau,\mu}}{1+K_{\tau,\mu}}\}$.

\item[(4)] Let $\tilde{C}_{0}$ be given by (1), then for every $u\in \mathcal{E}_{\tau,\mu}(\tilde C_0)$, $U(x,t)=U_{\tau,\mu,\tilde C_0}^-(x_\delta)$ is a sub-solution of \eqref{ODE2} on $\R\times \R$ for $0<\delta\ll 1$,
where $x_\delta=\underline{a}_{\mu,\tilde{\mu},d}+\delta$.
\end{itemize}
\end{tm}

To prove Theorem \ref{super-sub-solu-thm}, we first establish some estimates on
$V_{i}(\cdot;u)$ and $\frac{d}{dx}V_{i}(\cdot;u)$.

It follows from \eqref{Inverse of u} that
\begin{equation}\label{Estimates on Inverse of u}
\max\{\|V_{i}(\cdot;u)\|_{\infty}, \ \|\frac{d}{dx}V_{i}(\cdot;u)\|_{\infty} \}\leq \frac{\mu_i}{\lambda_i}\|u\|_{\infty}\quad \forall\ u\in C^{b}_{\rm unif}(\R).
\end{equation}
Furthermore, let
$$
C_{\rm unif}^{2,b}(\R)=\{u\in C_{\rm unif}^b(\R)\,|\, u^{'}(\cdot),\, u^{''}(\cdot)\in C_{\rm unif}^b(\R)\}.
$$
For every $i\in\{1,2\}$ and $u\in C_{\rm unif}^{b}(\R)$, $u\geq 0,$ we have that  $V_{i}(\cdot;u)\in C^{2,b}_{\rm unif}(\R)$ with $V_{i}(\cdot;u)\geq 0$ and
$$\|\frac{d^2}{dx^2} V_i(\cdot;u)\|_{\infty}=\|\lambda_iV_i(\cdot;u)-\tau c_\mu\frac{d}{dx} V(\cdot;u)-\mu_iu\|_{\infty}\leq \|\lambda_iV_i(\cdot;u)-\mu_iu\|_{\infty}+\tau c_\mu \|\frac{d}{dx}V(\cdot;u)\|_{\infty}.
$$
Combining this with inequality \eqref{Estimates on Inverse of u}, we obtain that
\begin{equation}\label{Estimates on Inverse of V}
\max\{\|V_i(\cdot;u)\|_{\infty}, \ \|\frac{d}{dx}V_i(\cdot;u)\|_{\infty}, \|\frac{d^2}{dx^2}V(\cdot;u)\|_{\infty}\}\leq (\tau c_\mu+\mu_i+\frac{\mu_i}{\lambda_i})\|u\|_{\infty}\quad \forall\ u\in \mathcal{E}_{\tau,\mu}(C_0), \ i=1,2.
\end{equation}

The next Lemma provides  pointwise and uniform estimates  for $(\chi_2\lambda_2V_2-\chi_1\lambda_1V_1)(\cdot;u)$ whenever $u\in \mathcal{E}_{\tau, \mu}(C_0)$.

\begin{lem}\label{Mainlem2}
For every $i\in\{1,2\}$, $C_0>0$, $\tau>0$, $0<\mu<\min\{\sqrt{a},\ \sqrt{\frac{\lambda_1+\tau a}{(1-\tau)_{+}}},\sqrt{\frac{\lambda_2+\tau a}{(1-\tau)_{+}}}\}$ and $u\in \mathcal{E}_{\tau,\mu}(C_0)$, let $V_i(\cdot;u)$ be defined as in \eqref{Inverse of u}, then
\begin{align}\label{Eq_MainLem2}
&\max\left\{0, \int_{0}^{\infty}\overline{M}(s)(e^{-(\tau\mu c_{\mu}-\mu^2)s}\varphi_{\tau,\mu}(x)-de^{-(\tau\tilde{\mu}c_{\mu}-\tilde{\mu}^2)s}\varphi_{\tau,\tilde{\mu}}(x))ds \right\} -\min\left\{ \underline{M}C_0, \underline{M}_{\tau,\mu}\varphi_{\tau,\mu}(x)\right\}\nonumber\\
&\leq (\chi_2\lambda_2V_2-\chi_1\lambda_1 V_1)(x;u)\leq\nonumber\\
& \min\left\{ \overline{M}C_0, \overline{M}_{\tau,\mu}\varphi_{\tau,\mu}(x)\right\}-\max\left\{0, \int_{0}^{\infty}\underline{M}(s)(e^{-(\tau\mu c_{\mu}-\mu^2)s}\varphi_{\tau,\mu}(x)-de^{-(\tau\tilde{\mu}c_{\mu}-\tilde{\mu}^2)s}\varphi_{\tau,\tilde{\mu}}(x))ds \right\}
\end{align}
\end{lem}
\begin{proof} For every  $u\in \mathcal{E}_{\tau,\mu}(C_0)$, since $ 0 \leq U^{-}_{\tau,\mu,C_0}\leq u\leq U^{+}_{\tau,\mu,C_0}$ then
for every $x\in \R$, setting $Z=\chi_2\lambda_2V_2-\chi_1\lambda_1 V_1$, we have that
\begin{align}\label{A-eq1}
Z(x;u)&=\int_0^{\infty}\int_{\R}\left(\chi_2\mu_2\lambda_2e^{-\lambda_2 s}-\chi_1\mu_1\lambda_1e^{-\lambda_2 s}\right)\frac{e^{-z^2}}{\sqrt{\pi}}u(x+2\sqrt{s}z+c_{\mu}\tau s)dzds\nonumber\\
&\leq \underbrace{\int_0^{\infty}\int_{\R}\overline{M}(s)\frac{e^{-z^2}}{\sqrt{\pi}}U^+_{\tau,\mu,C_0}(x+2\sqrt{s}z+c_{\mu}\tau s)dzds}_{M_1}\nonumber\\
&- \underbrace{\int_0^{\infty}\int_{\R}\underline{M}(s)\frac{e^{-z^2}}{\sqrt{\pi}}U^-_{\tau,\mu,C_0}(x+2\sqrt{s}z+c_{\mu}\tau s)dzds}_{M_2},
\end{align}
 and
\begin{align}\label{A-eq2}
Z(x;u)&\geq \underbrace{\int_0^{\infty}\int_{\R}\overline{M}(s)\frac{e^{-z^2}}{\sqrt{\pi}}U^{-}_{\tau,\mu,C_0}(x+2\sqrt{s}z+c_{\mu}\tau s)dzds}_{M_3}\nonumber\\
&- \underbrace{\int_0^{\infty}\int_{\R}\underline{M}(s)\frac{e^{-z^2}}{\sqrt{\pi}}U^{+}_{\tau,\mu,C_0}(x+2\sqrt{s}z+c_{\mu}\tau s)dzds}_{M_4}.
\end{align}

Observe that

\begin{align}\label{A-eq3}
M_{1}& \leq \min\left\{C_0\int_0^{\infty}\int_{\R}\overline{M}(s)\frac{e^{-z^2}}{\sqrt{\pi}}dzds, \int_0^{\infty}\int_{\R}\overline{M}(s)\frac{e^{-z^2}}{\sqrt{\pi}}\varphi_{\tau,\mu}(x+2\sqrt{s}z+\tau c_{\mu}s)dzds \right\}\nonumber\\
&= \min\left\{\overline{M}C_0,\varphi_{\tau,\mu}(x) \int_0^{\infty}\int_{\R}\overline{M}(s)\frac{e^{-z^2}}{\sqrt{\pi}}e^{-\mu(2\sqrt{s}z+\tau c_{\mu}s)}dzds \right\}\nonumber\\
&= \min\left\{\overline{M}C_0,\varphi_{\tau,\mu}(x) \int_0^{\infty}\underbrace{\overline{M}(s)e^{-(\tau\mu c_{\mu} -\mu^2c)s}}_{\overline{M}_{\tau,\mu}(s)}\underbrace{\left[\int_{\R}\frac{e^{-(z+\sqrt{s}\mu)^2}}{\sqrt{\pi}}dz\right]}_{=1}ds \right\}\nonumber\\
&=\min\left\{ \overline{M}C_0, \overline{M}_{\tau,\mu}\varphi_{\tau,\mu}(x)\right\},
\end{align}

\begin{align}\label{A-eq4}
M_{2}& \geq \max\left\{0, \int_0^{\infty}\int_{\R}\underline{M}(s)\frac{e^{-z^2}}{\sqrt{\pi}}(\varphi_{\tau,\mu}(x+2\sqrt{s}z+\tau c_{\mu}s)-d \varphi_{\tau,\tilde{\mu}}(x+2\sqrt{s}z+\tau c_{\mu}s) )dzds \right\}\nonumber\\
&= \max\left\{0,\int_0^{\infty}\underline{M}(s)\left[\varphi_{\tau,\mu}(x)\underbrace{\int_{\R}\frac{e^{-z^2}}{\sqrt{\pi}}e^{-\mu(2\sqrt{s}z+\tau c_{\mu}s)}dz}_{=e^{-(\tau\mu c_{\mu}-\mu^2)s}}-\varphi_{\tau,\tilde{\mu}}(x)\underbrace{\int_{\R}\frac{e^{-z^2}}{\sqrt{\pi}}e^{-\tilde{\mu}(2\sqrt{s}z+\tau c_{\mu}s)}dz}_{=e^{-(\tau\tilde{\mu}c_{\mu}-\tilde{\mu}^2)s}}\right]ds\right\}\nonumber\\
&= \max\left\{0, \int_{0}^{\infty}\underline{M}(s)(e^{-(\tau\mu c_{\mu}-\mu^2)s}\varphi_{\tau,\mu}(x)-de^{-(\tau\tilde{\mu}c_{\mu}-\tilde{\mu}^2)s}\varphi_{\tau,\tilde{\mu}}(x))ds \right\},
\end{align}

\begin{align}\label{A-eq5}
M_{3}& \geq \max\left\{0, \int_0^{\infty}\int_{\R}\overline{M}(s)\frac{e^{-z^2}}{\sqrt{\pi}}(\varphi_{\tau,\mu}(x+2\sqrt{s}z+\tau c_{\mu}s)-d \varphi_{\tau,\tilde{\mu}}(x+2\sqrt{s}z+\tau c_{\mu}s) )dzds \right\}\nonumber\\
%&= \max\left\{0,\int_0^{\infty}\overline{M}(s)\left[\varphi_{\tau,\mu}(x)\underbrace{\int_{\R}\frac{e^{-z^2}}{\sqrt{\pi}}e^{-\mu(2\sqrt{s}z+\tau c_{\mu}s)}dz}_{=e^{-(\tau\mu c_{\mu}-\mu^2)s}}-\varphi_{\tau,\tilde{\mu}}(x)\underbrace{\int_{\R}\frac{e^{-z^2}}{\sqrt{\pi}}e^{-\tilde{\mu}(2\sqrt{s}z+\tau c_{\mu}s)}dz}_{=e^{-(\tau\tilde{\mu}c_{\mu}-\tilde{\mu}^2)s}}\right]ds\right\}\nonumber\\
&= \max\left\{0, \int_{0}^{\infty}\overline{M}(s)(e^{-(\tau\mu c_{\mu}-\mu^2)s}\varphi_{\tau,\mu}(x)-de^{-(\tau\tilde{\mu}c_{\mu}-\tilde{\mu}^2)s}\varphi_{\tau,\tilde{\mu}}(x))ds \right\},
\end{align}
and
\begin{align}\label{A-eq6}
M_{4}& \leq \min\left\{C_0\int_0^{\infty}\int_{\R}\underline{M}(s)\frac{e^{-z^2}}{\sqrt{\pi}}dzds, \int_0^{\infty}\int_{\R}\underline{M}(s)\frac{e^{-z^2}}{\sqrt{\pi}}\varphi_{\tau,\mu}(x+2\sqrt{s}z+\tau c_{\mu}s)dzds \right\}\nonumber\\
%&= \min\left\{\underline{M}C_0,\varphi_{\tau,\mu}(x) \int_0^{\infty}\int_{\R}\underline{M}(s)\frac{e^{-z^2}}{\sqrt{\pi}}e^{-\mu(2\sqrt{s}z+\tau c_{\mu}s)}dzds \right\}\nonumber\\
%&= \min\left\{\underline{M}C_0,\varphi_{\tau,\mu}(x) \int_0^{\infty}\underbrace{\underline{M}(s)e^{-(\tau\mu c_{\mu} -\mu^2c)s}}_{\underline{M}_{\tau,\mu}(s)}\underbrace{\left[\int_{\R}\frac{e^{-(z+\sqrt{s}\mu)^2}}{\sqrt{\pi}}dz\right]}_{=1}ds \right\}\nonumber\\
&=\min\left\{ \underline{M}C_0, \underline{M}_{\tau,\mu}\varphi_{\tau,\mu}(x)\right\},
\end{align}
The Lemma follows from inequalities \eqref{A-eq1}, \eqref{A-eq2}, \eqref{A-eq3}, \eqref{A-eq4},\eqref{A-eq5}, and \eqref{A-eq6}.
\end{proof}

\medskip

Next, we present a pointwise/uniform  estimate  for $\frac{d}{dx}(\chi_1V_1-\chi_2V_2)(\cdot;u)$ whenever $u\in \mathcal{E}_{\tau,\mu}(C_0).$

\medskip

\begin{lem}\label{Mainlem3} Let $\tau>0,\ C_0>0$
 and $0<\mu<\min\{\sqrt{a}, \sqrt{\frac{\lambda_1+\tau a}{(1-\tau)_{+}}},\sqrt{\frac{\lambda_2+\tau a}{(1-\tau)_{+}}}\}$ be fixed. Let $u\in C^{b}_{\rm unif}(\R)$ and $V_{i}(\cdot;u)\in C^{2,b}_{\rm unif}(\R)$ be defined by \eqref{Inverse of u}, then
\begin{align}\label{Eq_Mainlem01}
&\left|\frac{d}{dx}(\chi_1V_1-\chi_2V_2)(x;u)\right|\nonumber\\
&\leq\min\left\{KC_0\ ,\ \varphi_{\tau,\mu}(x)\int_0^{\infty}\left|\mu_1\chi_1e^{-\lambda_1 s}-\chi_2\mu_2e^{-\lambda_2 s}\right|\frac{(1+
\mu\sqrt{\pi s})e^{-(\tau\mu c_{\mu}-\mu^2)s}}{\sqrt{\pi s}}ds\right\}
\end{align}
for every $x \in\R$ and every $u\in\mathcal{E}_{\tau,\mu}(C_0)$, .
\end{lem}

\begin{proof} Let $u\in\mathcal{E}_{\tau,\mu}(C_0)$ and fix any $x\in \R$.
\begin{align}\label{A-eq7}
\left|\frac{d}{dx}(\chi_1V_1-\chi_2V_2)(x;u)\right|&\leq \int_0^\infty\int_\R\left|\mu_1\chi_1e^{-\lambda_1 s}-\chi_2\mu_2e^{-\lambda_2 s}\right|\frac{|z|e^{-z^2}}{\sqrt{\pi s}}U^{+}_{\tau,\mu}(x+2\sqrt{s}z+\tau c_{\mu}s)dzds\nonumber\\
&\leq \varphi_{\tau,\mu}(x)\int_0^\infty\int_{\R}\left|\mu_1\chi_1e^{-\lambda_1 s}-\chi_2\mu_2e^{-\lambda_2 s}\right|\frac{|z|e^{-(z+\sqrt{s}\mu)^2}}{\sqrt{\pi s}}e^{-(\tau\mu c_{\mu}-\mu^2)s}dzds \nonumber\\
&\leq \varphi_{\tau,\mu}(x)\int_0^{\infty}\left|\mu_1\chi_1e^{-\lambda_1 s}-\chi_2\mu_2e^{-\lambda_2 s}\right|\frac{e^{-(\tau\mu c_{\mu}-\mu^2)s}}{\sqrt{\pi s}}\left[\int_{\R}|z|e^{-(z+\sqrt{s}\mu)^2}dz\right]ds\nonumber\\
&\leq \varphi_{\tau,\mu}(x)\int_0^{\infty}\left|\mu_1\chi_1e^{-\lambda_1 s}-\chi_2\mu_2e^{-\lambda_2 s}\right|\frac{(1+
\mu\sqrt{\pi s})e^{-(\tau\mu c_{\mu}-\mu^2)s}}{\sqrt{\pi s}}ds .
\end{align}
Note that we have used the following fact in the last inequality
\begin{eqnarray*}
\int_{\R}|z|e^{-|z-\mu\sqrt{s}|^2}dz  = \int_{\R}|z+\mu\sqrt{s}|e^{-|z|^2}dz \leq   \int_{\R}(|z|+\mu\sqrt{s})e^{-|z|^2}dz= 1+\mu\sqrt{\pi s}.
\end{eqnarray*}
On the other hand, we have that

\begin{align}\label{A-eq8}
\frac{d}{dx}(\chi_1V_1-\chi_2V_2)(x;u)&= \int_0^\infty\int_\R\left(\mu_1\chi_1e^{-\lambda_1 s}-\chi_2\mu_2e^{-\lambda_2 s}\right)\frac{ze^{-z^2}}{\sqrt{\pi s}}u(x+2\sqrt{s}z+\tau c_{\mu}s)dzds\nonumber\\
&=\int_0^\infty\int_{0}^{\infty}\left(\mu_1\chi_1e^{-\lambda_1 s}-\chi_2\mu_2e^{-\lambda_2 s}\right)\frac{ze^{-z^2}}{\sqrt{\pi s}}u(x+2\sqrt{s}z+\tau c_{\mu}s)dzds\nonumber\\
&-\int_0^\infty\int_{0}^{\infty}\left(\mu_1\chi_1e^{-\lambda_1 s}-\chi_2\mu_2e^{-\lambda_2 s}\right)\frac{ze^{-z^2}}{\sqrt{\pi s}}u(x-2\sqrt{s}z+\tau c_{\mu}s)dzds\nonumber\\
&\leq\int_0^\infty\int_{0}^{\infty}\left(\mu_1\chi_1e^{-\lambda_1 s}-\chi_2\mu_2e^{-\lambda_2 s}\right)_{+}\frac{ze^{-z^2}}{\sqrt{\pi s}}u(x+2\sqrt{s}z+\tau c_{\mu}s)dzds\nonumber\\
&+\int_0^\infty\int_{0}^{\infty}\left(\mu_1\chi_1e^{-\lambda_1 s}-\chi_2\mu_2e^{-\lambda_2 s}\right)_{-}\frac{ze^{-z^2}}{\sqrt{\pi s}}u(x-2\sqrt{s}z+\tau c_{\mu}s)dzds\nonumber\\
&\leq\int_0^\infty\int_{0}^{\infty}\left|\mu_1\chi_1e^{-\lambda_1 s}-\chi_2\mu_2e^{-\lambda_2 s}\right|\frac{ze^{-z^2}}{\sqrt{\pi s}}C_0dzds = KC_0.
\end{align}
Similarly, we have that
\begin{equation}\label{A-eq9}
\frac{d}{dx}(\chi_2V_2-\chi_1V_1)(x;u)\leq \int_0^\infty\int_{0}^{\infty}\left|\mu_2\chi_2e^{-\lambda_2 s}-\chi_1\mu_1e^{-\lambda_1 s}\right|\frac{ze^{-z^2}}{\sqrt{\pi s}}C_0dzds = KC_0.
\end{equation}

The Lemma follows from \eqref{A-eq7}, \eqref{A-eq8}, and \eqref{A-eq9}.
\end{proof}

\begin{rk} It follows from Lemma \ref{Mainlem2} that
\begin{equation}\label{unif-bound}
-\underline{M}C_0\leq (\chi_2\lambda_2V_2-\chi_1\lambda_1V_1)(x,u)\leq \overline{M}C_0,\quad \forall x\in\in\R, \ u\in \mathcal{E}_{\tau,\mu}(C_0)
\end{equation}
and
\begin{equation}\label{pointw-bound}
-\underline{M}_{\tau,\mu}\varphi_{\tau,\mu}(x)\leq (\chi_2\lambda_2V_2-\chi_1\lambda_1V_1)(x,u)\leq \overline{M}_{\tau,\mu}\varphi_{\tau,\mu}(x),\quad \forall x\in\in\R, \ u\in \mathcal{E}_{\tau,\mu}(C_0).
\end{equation}
\end{rk}

 \medskip

  Now we are ready to present the proof of Theorem \ref{super-sub-solu-thm}.

  \medskip
\begin{proof}[Proof of Theorem \ref{super-sub-solu-thm}]
 For every $U\in C^{2,1}(\R\times\R_{+})$, let \begin{align}\label{mathcal L}
\mathcal{L}U=&U_{xx}+(c_{\mu}+\frac{d}{dx}(\chi_2V_2-\chi_1 V_1)(\cdot;u))U_{x}+(a+(\chi_2\lambda_2V_2-\chi_1\lambda_1V_1)(\cdot,u))U\nonumber\\
&+(c_{\mu}\tau\frac{d}{dx}(\chi_1V_1-\chi_2V_2)(\cdot,u) -(b+\chi_2\mu_2-\chi_1\mu_1)U)U .
\end{align}
(1)  First, using inequality \eqref{unif-bound}, we have that
\begin{eqnarray}\label{N001}
\mathcal{L}( C_0)&=& (a+(\chi_2\lambda_2V_2-\chi_1\lambda_1V_1)(\cdot,u)+c_{\mu}\tau\frac{d}{dx}(\chi_1V_1-\chi_2V_2)(\cdot,u) -(b+\chi_2\mu_2-\chi_1\mu_1)C_0)C_0\nonumber\\
              & \leq & (a-(b+\chi_2\mu_2-\chi_1\mu_1-\overline{M}-c_{\mu}\tau K)C_0)C_0
\end{eqnarray}
Since $\chi_2\mu_2-\chi_1\mu_1-\overline{M}=-\underline{M}$, then it follows from {\bf (H)} that $b+\chi_2\mu_2-\chi_1\mu_1-\overline{M}-c_{\mu}\tau K$.  Thus taking $\tilde{C}_{0}:=\frac{a}{b+\chi_2\mu_2-\chi_1\mu_1-\overline{M}-c_{\mu}\tau K}$, it follows from inequality \eqref{N001} that for every $C_{0}\geq \tilde{C}_0$, we have that
$$
\mathcal{L}(C_{0})\leq 0.
$$ Hence, for every $C_{0}\geq \tilde{C}_0$, we have that $U(x,t)=C_0$ is a super-solution of \eqref{ODE2} on $\R\times\R$.\\
(2) It follows from Lemma \ref{Mainlem3}, and inequality \eqref{pointw-bound}, and \eqref{Eq01_Th1} that
\begin{align*}
 \mathcal{L}(\varphi_{\tau,\mu}) =& \varphi''_{\tau,\mu}(x)+(c_{\mu}+\frac{d}{dx}(\chi_2V_2-\chi_1V_1)(\cdot;u))\varphi_{\tau,\mu}'(x)+(a+(\chi_2\lambda_2V_2-\chi_1\lambda_1 V_1)(\cdot;u))\varphi_{\tau,\mu}\nonumber\\
  & + (\tau c_\mu \frac{d}{dx}(\chi_1V_1-\chi_2V_2)(\cdot,u)-(b+\chi_2\mu_2-\chi_1\mu_1)\varphi_{\tau,\mu})\varphi_{\tau,\mu}\nonumber\\
  =& \underbrace{(\varphi''_{\tau,\mu}+c_{\mu}\varphi'_{\tau,\mu}+a\varphi_{\tau,\mu})}_{=0} +\frac{d}{dx}(\chi_2V_2-\chi_1V_1)(\cdot;u)\varphi_{\tau,\mu}'\nonumber\\
  & +((\chi_2\lambda_2V_2-\chi_1\lambda_1 V_1)(\cdot;u)+\tau c_\mu \frac{d}{dx}(\chi_1V_1-\chi_2V_2)(\cdot,u)-(b+\chi_2\mu_2-\chi_1\mu_1)\varphi_{\tau,\mu})\varphi_{\tau,\mu}\nonumber\\
 = & \left((\tau c_{\mu}+\mu) \frac{d}{dx}(\chi_1V_1-\chi_2V_2)(\cdot;u)+(\chi_2\lambda_2V_2-\chi_1\lambda_1 V_1)(\cdot;u)\right)\varphi_{\tau,\mu}\nonumber\\
  & -(b+\chi_2\mu_2-\chi_1\mu_1)\varphi_{\tau,\mu}^{2}\nonumber\\
  \leq & - \left(b+\chi_2\mu_2-\chi_1\mu_1-(\tau c_{\mu}+\mu)K_{\tau,\mu}-\overline{M}_{\tau,\mu}\right)\varphi_{\tau,\mu}^{2}\leq 0.
\end{align*}
Hence $U(x,t)=\varphi_{\tau,\mu}(x)$ is also a super-solution of \eqref{ODE2} on $\R\times\R$.

(3)  Let $C_0>0$ and $O=(\underline{a}_{\mu,\tilde{\mu},d},\infty)$. Then for $x\in O$, $U_{\tau,\mu,C_0}^-(x)>0$.
 For $x\in O$, it follows from Lemma \ref{Mainlem3} , and inequality \eqref{pointw-bound}, and \eqref{Eq01_Th1} that
\begin{align*}
\mathcal{L}(U_{\tau,\mu,C_0}^{-}
) = & \mu^2\varphi_{\tau,\mu}-\tilde{\mu}^2d\varphi_{\tau,\tilde{\mu}} +(c_{\mu}+\frac{d}{dx}(\chi_2V_2-\chi_1V_1)(\cdot;u))(-\mu\varphi_{\tau,\mu}+d\tilde{\mu}\varphi_{\tau,\tilde{\mu}})\nonumber\\
& +(a+(\chi_2\lambda_2V_2-\chi_1\lambda_1V_1)(\cdot,u)+c_{\mu}\tau\frac{d}{dx}(\chi_1V_1-\chi_2V_2)(\cdot;u))U_{\tau,\mu,C_0}^{-} \nonumber\\
& -(b+\chi_2\mu_2-\chi_1\mu_1)\left( U_{\tau,\mu,C_0}^{-}\right)^2\nonumber\\
=& \underbrace{(\mu^{2}-\mu c_{\mu}+a)}_{=0}\varphi_{\tau,\mu} +d\underbrace{(\tilde{\mu}c_{\mu}-\tilde{\mu}^{2}-a)}_{A_{0}}\varphi_{\tau,\tilde{\mu}}\nonumber\\
&  +\frac{d}{dx}(\chi_1V_1-\chi_2V_2)(\cdot;u)((c_{\mu}\tau+\mu)\varphi_{\tau,\mu}-d(c_{\mu}\tau+\tilde{\mu})\varphi_{\tau,\tilde{\mu}})\nonumber\\
&  + ( (\chi_2\lambda_2V_2-\chi_1\lambda_1V_1)(\cdot,u) -(b+\chi_2\mu_2-\chi_1\mu_1)U_{\tau,\mu,C_0}^{-})U_{\tau,\mu,C_0}^{-}\nonumber\\
 \geq & dA_{0}\varphi_{\tau,\tilde{\mu}} -K_{\tau,\mu}((\tau c_\mu+\mu)\varphi_{\tau,\mu}+d(\tau c_\mu+\tilde{\mu})\varphi_{\tau,\tilde{\mu}})\varphi_{\tau,\mu}\nonumber\\
&  -\left(\underline{M}_{\tau,\mu}\varphi_{\tau,\mu}+(b+\chi_2\mu_2-\chi_1\mu_1)U^{-}_{\tau,\mu,C_0}\right)U^{-}_{\tau,\mu,C_0}\nonumber\\
=& dA_0\varphi_{\tau,\tilde{\mu}}-\underbrace{\left( (c_{\mu}\tau+\mu)K_{\tau,\mu}+\underline{M}_{\tau,\mu}+b+\chi_2\mu_2-\chi_1\mu_1\right)}_{A_1}\varphi_{\tau,\mu}^2 \nonumber\\
&  +d\left(2(b+\chi_2\mu_2-\chi_1\mu_1)-(\tau c_{\mu}+\tilde{\mu})K_{\tau,\mu} \right)\varphi_{\tau,\mu}\varphi_{\tau,\tilde{\mu}}-d^2(b+\chi_2\mu_2-\chi_1\mu_1)\varphi^2_{\tau,\tilde{\mu}}.
\end{align*}
Note that  $U_{\tau,\mu,C_0}^{-}(x)>0$ is equivalent to $\varphi_{\tau,\mu}(x)>d\varphi_{\tau,\tilde{\mu}}(x)$, which is again equivalent to
$$
d(b+\chi_2\mu_2-\chi_1\mu_1)\varphi_{\tau,\mu}(x)\varphi_{\tau,\tilde{\mu}}(x)>d^{2}(b+\chi_2\mu_2-\chi_1\mu_1)\varphi^2_{\tau,\tilde{\mu}}(x).
$$
Since $A_{1}>0$, thus for $x\in O$, we have
\begin{eqnarray*}
\mathcal{L}U_{\mu}^{-}(x) & \geq &  dA_{0}\varphi_{\tau,\tilde{\mu}}(x) -A_{1}\varphi_{\tau,\mu}^2(x)\nonumber\\
& & +d\underbrace{\left(b+\chi_2\mu_2-\chi_1\mu_1-(\tau c_{\mu}+\tilde{\mu})K_{\tau,\mu}\right)}_{A_{2}}\varphi_{\tau,\mu}(x)\varphi_{\tau,\tilde{\mu}}(x)\nonumber\\
& =& A_{1}\left(\frac{dA_{0}}{A_{1}}e^{(2\mu-\tilde{\mu})x}-1\right)\varphi_{\tau,\mu}^{2}(x) +dA_{2}\varphi_{\tau,\mu}(x)\varphi_{\tau,\tilde{\mu}}(x).
\end{eqnarray*}
Note also that, by \eqref{Eq01_Th1},
\begin{eqnarray}\label{Eq1 of Th2}
A_{2}= \left(b+\chi_2\mu_2-\chi_1\mu_1-(\tau c_{\mu}+\mu)K_{\tau,\mu}\right) -(\tilde{\mu}-\mu)K_{\tau,\mu}\ge 0
\end{eqnarray}
whenever $(\tilde{\mu}-\mu)K_{\tau,\mu}\leq b+\chi_2\mu_2-\chi_1\mu_1-(\tau c_{\mu}+\mu)K_{\tau,\mu}$.
Observe that
$$
A_{0}=\frac{(\tilde{\mu}-\mu)(a-\mu\tilde{\mu})}{\mu}>0,\quad \forall\ 0<\mu<\tilde{\mu}<\sqrt{a}.
$$
 Furthermore, we have that $U_{\tau,\mu,C_0}^{-}(x)>0$ implies that $x>0$ for $d\geq \max\{1,C_0^{\frac{\mu-\tilde{\mu}}{\mu}}\}$. Thus, for every $ d\geq d_{0}:= \max\{1, \frac{A_{1}}{A_{0}}, C_0^{\frac{\mu-\tilde{\mu}}{\mu}}\}$, we have that
\begin{equation}\label{E1}
\mathcal{L}U_{\tau,\mu,C_0}^{-}(x) > 0
\end{equation}
whenever $x\in O$ and $\mu<\tilde{\mu}< \min\{\sqrt{a},\sqrt{\frac{\lambda_1+\tau a}{(1-\tau)_{+}}},\sqrt{\frac{\lambda_2+\tau a}{(1-\tau)_{+}}},2\mu,\mu+\frac{b+\chi_2\mu_2-\chi_1\mu_1-(\tau c_{\mu}+\mu)K_{\tau,\mu}}{1+K_{\tau,\mu}}\}$. Hence $U(x,t)=U_{\tau,\mu,C_0}^-(x)$ is a sub-solution of \eqref{ODE2} on $(\underline{a}_{\mu,\tilde{\mu},d},\infty)\times\R$.

(4) Since $\overline{M}-\underline{M}=\chi_2\mu_2-\chi_1\mu_1$,  thus, it follows from \eqref{Eq01_Th1} that
\begin{align}
a -\underline{M}\tilde{C}_0-\tau c_{\mu}K\tilde{C}_0&=a\left( 1-\frac{\underline{M}+\tau c_{\mu}K}{b+\chi_2\mu_2-\chi_1\mu_1-\overline{M}-\tau c_{\mu}K}\right)\nonumber\\
&=\frac{a(b+2\chi_2\mu_2-2\chi_1\mu_1-2\overline{M}-2\tau c_{\mu}K)}{b+\chi_2\mu_2-\chi_1\mu_1-\overline{M}-\tau c_{\mu}K}>0.
\end{align}
Hence, for $0<\delta\ll 1$, we have that
\begin{small}
\begin{align*}
\mathcal{L}(U_{\tau,\mu,\tilde{C}_0})&=(a+(\chi_2\lambda_2V_2-\chi_1\lambda_1V_1)(\cdot,u) +c_{\mu}\tau\frac{d}{dx}(\chi_1V_1-\chi_2V_2)(\cdot;u))U_{\tau,\mu,\tilde{C}_0}^-(x_\delta)\nonumber\\
& -(b+\chi_2\mu_2-\chi_1\mu_1)\left[U_{\tau,\mu,\tilde{C}_0}^-(x_\delta)\right]^2\nonumber\\
& \geq (a -\underline{M}\tilde{C}_0-\tau c_{\mu}K\tilde{C}_0 -(b+\chi_2\mu_2-\chi_1\mu_1)U_{\tau,\mu,\tilde{C}_0}^-(x_\delta))U_{\tau,\mu,\tilde{C}_0}^-(x_\delta)
\end{align*}
\end{small}
where $x_\delta=\underline{a}_{\mu,\tilde{\mu},d}+\delta$. This implies that $U(x,t)=U_{\tau,\mu,\tilde{C}_0}^-(x_\delta)$ is
a sub-solution of \eqref{ODE2} on $\R\times\R$.
\end{proof}

\section{Traveling wave solutions}

In this section we study the existence and nonexistence  of traveling wave solutions of \eqref{Main-eq1} connecting $(\frac{a}{b},\frac{a\mu_1}{b\lambda_1},\frac{a\mu_2}{b\lambda_2})$ and
$(0,0,0)$, and prove Theorems C and D.

\subsection{Proof of Theorem  C}

In this subsection, we  prove Theorem C. To this end,  we first prove  the following important result.

\begin{tm}\label{existence-tv-thm}
 Assume (H).  Then \eqref{Main-eq1} has a traveling wave solution
$(u(x,t),v_1(x,t),v_2(x,t))=(U(x-c_\mu t),V_1(x-c_\mu t),V_2(x-c_{\mu}t))$ satisfying
$$
\lim_{x\to-\infty}U(x)=\frac{a}{b} \quad \text{and}\quad
\lim_{x\to\infty}\frac{U(x)}{e^{- \mu x}}=1
$$
where $c_{\mu}=(\mu+\frac{a}{\mu})$.
\end{tm}

%Our key idea to prove the above theorem is to prove that, for any  $0<\mu<\min\{1,  \sqrt{\frac{\lambda_1}{a}}, \sqrt{\frac{\lambda_2}{a}}\}$ , and  $M+N+\chi_1\mu_1<b+\chi_2\mu_2$ satisfy \eqref{sup-sub-solu-eq}, there is $u^*(\cdot)\in\mathcal{E}_{\mu}$ such that  $(U(\cdot),V_1(\cdot),V_2(\cdot))=(u^*(\cdot),V_1(\cdot;u^*),V_2(\cdot;u^*))$ is a stationary solution of \eqref{Main-eq2} with $c=c_\mu$, where  $V_{i}(\cdot;u^*)$ is given by \eqref{Inverse of u}, and $u^*(-\infty)=\frac{a}{b}$ and $u^*(\infty)=0$, which implies that $(u(x,t),v_1(x,t),v_2(x,t))=(u^*(x-c_\mu t),V_1(x-c_\mu t;u^*),V_2(x-c_\mu t;u^*))$ is a traveling wave solution of \eqref{main-eq1} connecting $(\frac{a}{b},\frac{a\mu_1}{b\lambda_1},\frac{a\mu_2}{\lambda_2})$ and $(0,0,0)$ with speed $c=c_\mu$.

In order to prove Theorem \ref{existence-tv-thm}, we first prove some lemmas. These Lemmas extend some of the results established in \cite{SaSh2}, so some details might be omitted in their proofs. The reader is referred to the proofs of Lemmas 3.2, 3.3, 3.5 and 3.6 in \cite{SaSh2} for more details.

 In the remaining part of this subsection we shall suppose that {\bf (H)} holds and $\tilde \mu$ is fixed, where $\tilde \mu$ satisfies
   $$\mu<\tilde{\mu}<\min\{\sqrt{a},\sqrt{\frac{\lambda_1+
   \tau a}{(1-\tau)_{+}}},\sqrt{\frac{\lambda_1+\tau a}{(1-\tau)_+}},2\mu, \mu+\frac{b+\chi_2\mu_2-\chi_1\mu_1-(\tau c_{\mu}+\mu)_{\tau,\mu}}{1+K_{\tau,\mu}}\}.
    $$
 Furthermore, we choose $\tilde{C}_{0}=\tilde{C}_{0}(\tau,\chi_1,\mu_1,\lambda_1,\chi_2,\mu_2,\lambda_2,\mu)$ and  $d=d_0(\tau,\chi_1,\mu_1,\lambda_1,\chi_2,\mu_2,\lambda_2,\mu)$ to be the constants given by Theorem \ref{super-sub-solu-thm} and to be fixed and set $U^{+}_{\tau,\mu,\tilde{C}_0}:=U^{+}_{\mu}$ and $U^{-}_{\tau,\mu,\tilde{C}_0}:=U^{-}_{\mu}$. Fix $u\in\mathcal{E}_{\tau,\mu}(\tilde{C}_{0})$. For given $u_0\in C_{\rm unif}^b(\R)$, let
  $U(x,t;u_0,u)$ be the solution of \eqref{ODE2} with
$U(x,0;u_0,u)=u_0(x)$. By the arguments in the proofs of Theorem 1.1 and Theorem 1.5 in \cite{SaSh1}, we have $U(x,t; U_{\mu}^+,u)$ exists for all $t>0$ and
${ U(\cdot,\cdot;{ U_{\mu}^+},u)}\in C([0,\infty),C^{b}_{\rm unif}(\R))\cap C^{1}((0\ ,\ \infty),C^{b}_{\rm unif}(\R))\cap C^{2,1}(\R\times(0,\ \infty))$ satisfying
\begin{equation}
U(\cdot,\cdot; U_{\mu}^+,u), U_{x}(\cdot,\cdot; U_{\mu}^+,u),U_{xx}(\cdot,t; U_{\mu}^+,u),U_{t}(\cdot,\cdot; U_{\mu}^+,u)\in  C^{\theta}((0, \infty),C_{\rm unif}^{\nu}(\R))
\end{equation}
for $0<\theta, \nu \ll 1$.

\begin{lem} \label{lm1} Assume (H).
 Then for every $u\in \mathcal{E}_{\tau,\mu}(\tilde{C}_{0})$, the following hold.
\begin{description}
\item[(i)] $0\leq U(\cdot,t; U_{\mu}^+,u)\leq U_{ \mu}^{+}(\cdot)$ for every $t\geq 0.$
\item[(ii)] $U(\cdot,t_{2}; U_{\mu}^+,u)\leq U(\cdot,t_{1}; U_{\mu}^+, u) $ for every $0\leq t_{1}\leq t_{2}$.
\end{description}
\end{lem}
\begin{proof}
(i)   Note that $0\leq U^{+}_{\mu}(\cdot)\leq \tilde{C}_{0}$. Then by
comparison principle for parabolic equations and Theorem \ref{super-sub-solu-thm}(1), we have
\begin{equation*}
0\leq U(x,t; U_{\mu}^+,u)\leq \tilde{C}_{0} \quad \forall\ x\in\R,\ t\geq 0.
\end{equation*}

Similarly, note that $0\leq U_{\mu}^+(x)\le\varphi_{\mu}(x)$.  Then by  comparison principle for parabolic equations and
Theorem \ref{super-sub-solu-thm}(2)  again, we have
\begin{equation*}
U(x,t;U_{\mu}^+,u)\leq \varphi_{\mu}(x) \ \quad \forall\ x\in\R,\ t\geq 0.
\end{equation*}
Thus $U(\cdot,t;U_{\mu}^+,u)\leq U^{+}_{\mu}$. This completes the proof of (i).

(ii)  For $0\leq t_{1}\leq t_{2}$, since
$$
U(\cdot,t_{2};U_{\mu}^+,u)=U(\cdot,t_{1} ;U(\cdot,t_{2}-t_{1};U_{\mu}^+,u),u)
$$
and by (i), $U(\cdot,t_{2}-t_{1};U_{\mu}^+,u)\leq U^{+}_{\mu} $, (ii) follows from comparison principle for parabolic equations.
\end{proof}

Let us define $U(x; u)$ to be
\begin{equation}
\label{U-eq}
{
U(x; u)=\lim_{t\to\infty}U(x,t; U^{+}_{\mu}, u)=\inf_{t>0}U(x,t; U^{+}_{\mu}, u).
}
\end{equation}
 By the a priori estimates for parabolic equations, the limit in \eqref{U-eq} is uniform in $x$ in compact subsets of $\R$
and $U(\cdot;u)\in C_{\rm unif}^b(\R)$.
Next we prove that the function $u\in\mathcal{E}_{\tau,\mu}(\tilde{C}_{0})\to U(\cdot;u)\in\mathcal{E}_{\tau,\mu}(\tilde{C}_{0})$.

\begin{lem}\label{lm2}
Assume (H).  Then,
\begin{equation}
U(x;u)\geq\begin{cases}  U^{-}_{\mu}(x),\quad x\ge \underline{a}_{\mu,\tilde{\mu},d}\cr
U_{\mu}^-(x_\delta),\quad x\le x_\delta=\underline{a}_{\mu,\tilde{\mu},d}+\delta
\end{cases}
\end{equation}\label{Eq2 of Th2}
for every $u\in\mathcal{E}_{\tau,\mu}(\tilde{C}_{0})$,  $x\in\R $, and $0<\delta\ll 1$.
\end{lem}

\begin{proof} Let $u\in \mathcal{E}_{\tau,\mu}(\tilde{C}_{0})$ be fixed.  Let $O=(\underline{a}_{\mu,\tilde{\mu},d},\infty)$.
Note that $U_{\mu}^-(\underline{a}_{\mu,\tilde{\mu},d})=0$. By Theorem \ref{super-sub-solu-thm}(3),
$U_{\mu}^-(x)$ is a sub-solution of \eqref{ODE2} on $O\times (0,\infty)$.
Note also that $U_{\mu}^+(x)\ge U_{\mu}^-(x)$ for $x\ge \underline{a}_{\mu,\tilde{\mu},d}$ and $U(\underline{a}_{\mu,\tilde{\mu},d},t;U_{\mu}^+, u)>0$
for all $t\ge 0$. Then by comparison principle for parabolic equations, we have that
$$
U(x,t;U_{\mu}^+,u)\ge U_{\mu}^-(x)\quad \forall \,\, x\ge \underline{a}_{\mu,\tilde{\mu},d},\,\, t\ge 0.
$$

Now for any $0<\delta\ll 1$, by Theorem \ref{super-sub-solu-thm}(4), $U(x,t)=U_{\mu}^-(x_\delta)$ is a sub-solution of
\eqref{ODE2} on $\R\times \R$. Note that $U_{\mu}^+(x)\ge U_\mu^-(x_\delta)$ for $x\le x_\delta$ and
$U(x_\delta,t;U_{\mu}^+,u)\ge U_{\mu}^-(x_\delta)$ for $t\ge 0$. Then by comparison principle for parabolic equations again,
$$
U(x,t;U_{\mu}^+,u)\ge U_{\mu}^-(x_\delta)\quad \forall\,\, x\le x_\delta,\, \, t>0.
$$
The lemma then follows.
\end{proof}

\begin{rk}\label{Remark-lower-bound-for -solution}
 It follows from Lemmas \ref{lm1} and \ref{lm2} that if {\bf (H)} holds, then
$$
U_{\mu,\delta}^{-}(\cdot)\leq U(\cdot,t;U_{\mu}^+,u)\leq U^{+}_{\mu}(\cdot)$$
for every $u\in\mathcal{E}_{\tau,\mu}(\tilde{C}_0)$, $t\geq0$ and $0\le \delta\ll 1$, where
$$
U_{\mu,\delta}^-(x)=\begin{cases}  U^{-}_{\mu}(x),\quad x\ge \underline{a}_{\mu,\tilde{\mu},d}+\delta\cr
U_\mu^-(x_\delta),\quad x\le x_\delta=\underline{a}_{\mu,\tilde{\mu},d}+\delta.
\end{cases}
$$
 This implies that $$
U_{\mu,\delta}^{-}(\cdot)\leq U(\cdot;u)\leq U^{+}_{\mu}(\cdot)$$
for every $u\in\mathcal{E}_{\tau,\mu}(\tilde{C}_0)$. Hence  $u\in\mathcal{E}_{\tau,\mu}(\tilde{C}_0)\mapsto U(\cdot;u)\in \mathcal{E}_{\tau,\mu}(\tilde{C}_0).$
\end{rk}

\begin{lem}\label{MainLem02}
Assume (H).
 Then for every $\mathcal{E}_{\tau,\mu}(\tilde{C}_0)$ the associated function $U(\cdot;u)$ satisfied the elliptic equation,
\begin{align}\label{Eq_MainLem02}
0=&U_{xx}+(c_{\mu}+(\chi_2V_2-\chi_1V_1)_{x}(x;u))U_{x}+(a+(\chi_2\lambda_2V_2-\chi_1\lambda_1V_1)(x;u))U\nonumber\\
&+(c\tau(\chi_1V_1-\chi_2V_2)_{x}(x,u)-(b+\chi_2\mu_2-\chi_1\mu_1)U)U, \quad x\in \R,
\end{align}

\end{lem}

\begin{proof} The following arguments generalized the arguments used in the proof of Lemma 4.6 in \cite{SaSh2}. Hence we refer to \cite{SaSh2} for the proofs of the estimates stated below.

 Let $\{t_{n}\}_{n\geq 1}$ be an increasing sequence of positive real numbers converging to $\infty$. For every $n\geq 1$, define $U_{n}(x,t)=U(x,t+t_{n}; U_{\mu}^+, u)$ for every $x\in\R, \ t\geq 0$.
For every $n$, $U_{n}$ solves the PDE
\begin{equation*}
\begin{cases}
\partial_{t}U_{n}=\partial_{xx}U_{n}+(c_{\mu}+\partial_x(\chi_2V_2-\chi_1V_1)(\cdot;u))\partial_{x}U_{n}+(a+(\chi_2\lambda_2V_2-\chi_1\lambda_1V_1)(\cdot;u))U_{n}\\
\qquad \qquad +(c\tau(\chi_1V_1-\chi_2V_2)_{x}(\cdot,u)-(b+\chi_2\mu_2-\chi_1\mu_1)U_{n})U_{n}\\
U_{n}(\cdot,0)=U(\cdot,t_{n}; U_{\mu}^+, u).
\end{cases}
\end{equation*}

Let $\{T(t)\}_{t\geq 0}$ be the analytic semigroup on $C^{b}_{\rm unif}(\R)$ generated by $\Delta-I$
 and  let $X^{\beta}={\rm Dom}((I-\Delta)^{\beta})$ be the fractional power spaces of $I-\Delta$ on $C_{\rm unif}^b(\R)$ ($\beta\in [0,1]$).

 The variation of constant formula and the fact that $\partial_{xx}V_{i}(\cdot;u)-\lambda_{i}V=-\mu_i u$ yield that
\begin{eqnarray}\label{variation -of-const}
& &U(\cdot,t;U_{\mu}^+, u)\nonumber\\
&=& \underbrace{T(t)U_{\mu}^{+}}_{I_{1}(t)}+ \underbrace{\int_{0}^{t}T(t-s)(((c_{\mu}+\partial_{x}(\chi_2V_2-\chi_1V_1)(\cdot;u))U(\cdot,s; U_{\mu}^+, u))_{x})(s)ds}_{I_{2}(t)}\nonumber \\
& +&\underbrace{\int_{0}^{t}T(t-s)(1+a+(\chi_2\mu_2-\chi_1\mu_1) u)U(\cdot,s;U_{\mu}^+, u)ds}_{I_{3}(t)}\nonumber\\
&-&(b+\chi_2\mu_2-\chi_1\mu_1)\underbrace{\int_{0}^{t}T(t-s)U^{2}(\cdot,s;U_{\mu}^+, u)ds}_{I_{4}(t)}.\nonumber\\
\end{eqnarray}
Let $0<\beta<\frac{1}{2}$ be fixed. There is a positive constant $C_{\beta}$,  (see \cite{Dan Henry}), such that
\begin{equation*}
\|I_{1}(t)\|_{X^{\beta}}\leq \tilde{C}_0C_\beta t^{-\beta}e^{-t},
\end{equation*}
\begin{equation*}\|I_{2}(t)\|_{X^{\beta}}\leq  C_{\beta}\tilde{C}_0(c_{\mu}+K\tilde{C}_0)\Gamma(\frac{1}{2}-\beta),
\end{equation*}
\smallskip
$$
\|I_{3}(t)\|_{X^{\beta}}
 \leq  C_{\beta}\tilde{C}_{0}(a+1+|\chi_2\mu_2-\chi_1\mu_1|\tilde{C}_{0})\Gamma(1-\beta),
 $$
 and
 $$ \ \ \|I_{4}(t)\|_{X^{\beta}}\leq C_{\beta}\tilde{C}_{0}^{2}\Gamma(1-\beta).
$$
Note that we have used Lemma \ref{Mainlem3}, mainly  the fact that $|\partial_{x}(\chi_2V_{2}-\chi_1V_1)(\cdot;u)|\leq K\tilde{C}_{0}$, to obtain the uniform upper bound estimates for $\|I_{2}(t)\|_{X^{\beta}}$. Therefore, for every $T>0$ we have that
\begin{equation}\label{Eq_Convergence01}
\sup_{t\geq T}\|U(\cdot,t;U_{\mu}^+,u)\|_{X^{\beta}}\leq M_{T}<\infty,
\end{equation}
where
\begin{equation}\label{Eq_Conv02}
M_{T,\mu}= C_{\beta}\tilde{C}_{0}\Big[\frac{T^{-\beta}}{e^{T}}+  (c_{\mu}+K\tilde{C}_{0})\Gamma(\frac{1}{2}-\beta)+  (a+1+|\chi_2\mu_2+\chi_1\mu_1|+\tilde{C}_{0}\Gamma(1-\beta)\Big].
\end{equation}
Hence, it follows  from \eqref{Eq_Convergence01} that
\begin{equation}\label{Eqq000}
\sup_{n\geq 1, t\geq 0}\|U_{n}(\cdot,t)\|_{X^{\beta}}\leq M_{t_{1}}<\infty.
\end{equation}
Next, for every $t,h\geq 0$ and $n\geq 1$, we have that
\begin{equation}\label{Eqq00}
\|I_{1}(t+h+t_{n})-I_{1}(t+t_{n})\|_{X^{\beta}}\leq C_{\beta}h^{\beta}(t+t_{n})^{-\beta}e^{-(t+t_n)}\|U_{\mu}^{+}\|_{\infty}\leq C_{\beta}h^{\beta}t_{1}^{-\beta}e^{-t_1}\|U_{\mu}^{+}\|_{\infty},
\end{equation}
\begin{align}\label{Eqq02}
\|I_{2}(t+t_n+h)-I_{2}(t+t_n)\|_{X^{\beta}}
&\leq C_{\beta}\tilde{C}_0(c_\mu+K\tilde{C}_0)\Big[h^{\beta}\Gamma(\frac{1}{2	}-\beta)+\frac{h^{\frac{1}{2}-\beta}}{\frac{1}{2}-\beta} \Big] ,
\end{align}
\begin{align}\label{Eqq01}
\|I_{3}(t+h+t_n)-I_{3}(t+t_n)\|_{X^{\beta}}
&\leq C_{\beta}\tilde{C}_{0}(a+1+|\chi_2\mu_2-\chi_1\mu_1|\tilde{C}_{0})\Gamma(1-\beta)\Big[h^{\beta}\Gamma(1-\beta)+\frac{h^{1-\beta}}{1-\beta} \Big] ,
\end{align}
and
\begin{eqnarray}\label{Eqq03}
\|I_{4}(t+t_n+h)-I_{4}(t+t_n)\|_{X^{\beta}}\leq C_{\beta}\tilde{C}_{0}^2\Big[h^{\beta}\Gamma(1-\beta)+\frac{h^{1-\beta}}{1-\beta} \Big].
\end{eqnarray}
It follows from inequalities \eqref{Eqq000}, \eqref{Eqq00}, \eqref{Eqq01}, \eqref{Eqq02} and \eqref{Eqq03}, the functions $U_{n} : [0, \infty)\to X^{\beta}$ are uniformly bounded and equicontinuous.
Since $X^{\beta}$ is continuously imbedded in $C^{\nu}(\R)$ for every $0\leq \nu<2\beta$ (See \cite{Dan Henry}),
% then
%$$
%\sup_{n}\|U_{n}(\cdot,0)\|_{C^{\nu}_{\rm unif}}<\infty
%$$
%for every $0\leq \nu<{ 1}$.
therefore, the Arzela-Ascoli Theorem and Theorem 3.15 in   \cite{Friedman}, imply that there is a function $\tilde{U}(\cdot,\cdot;u)\in C^{2,1}(\R\times(0,\infty))$ and a subsequence $\{U_{n'}\}_{n\geq 1}$ of $\{U_{n}\}_{n\geq 1}$ such that $U_{n'}\to \tilde{U}$ in $C^{2,1}_{loc}(\R\times(0, \infty))$ as $n\to \infty$ and $\tilde{U}(\cdot,\cdot;u)$ solves the PDE
$$
\begin{cases}
\partial_{t}\tilde{U}=\partial_{xx}\tilde{U}+(c_{\mu}+\partial_{x}(\chi_2 V_2-\chi_1V_1)(\cdot;u)\partial_{x}\tilde{U}+(a+(\chi_2\lambda_2V_2-\chi_1\lambda_1V_1)(\cdot;u))\tilde{U}\\
\quad \qquad+(c\tau(\chi_1V_1-\chi_2V_2)(\cdot,u)-(b+\chi_2\mu_2-\chi_1\mu_1)\tilde{U})\tilde{U}, \ t>0\\
\tilde{U}(x,0)=\lim_{n\to \infty}U(x,t_{n'}; U_{\mu}^+, u).
\end{cases}
$$
But $U(x;u)=\lim_{t\to \infty}U(x,t; U_{\mu}^+, u)$ and $t_{n'}\to \infty$ as $n\to \infty$, hence $\tilde{U}(x,t;u)=U(x;u)$ for every $x\in \R,\ t\geq 0$. Hence $U(\cdot;u)$ solves \eqref{Eq_MainLem02}.
\end{proof}

\begin{lem}
\label{aux-lm}  Assume (H).
  Then, for any given $u\in\mathcal{E}_{\tau,\mu}(\tilde{C}_0)$,
\eqref{Eq_MainLem02} has a unique bounded non-negative solution satisfying that
\begin{equation}
\label{aux-eq1}
\liminf_{x\to -\infty}U(x)>0\quad {\rm and}\quad \lim_{x\to\infty}\frac{U(x)}{e^{-\mu x}}=1.
\end{equation}
\end{lem}
The proof of Lemma \ref{aux-lm} follows from  \cite[Lemma 3.6]{SaSh2}.

We now prove Theorem \ref{existence-tv-thm}.

\begin{proof}[Proof of Theorem  \ref{existence-tv-thm}]
Following the proof of Theorem 3.1 in \cite{SaSh2},
 let us consider the normed linear space  $\mathcal{E}=C^{b}_{\rm unif}(\R)$ endowed with the norm
$$\|u\|_{\ast}=\sum_{n=1}^{\infty}\frac{1}{2^n}\|u\|_{L^{\infty}([-n,\ n])}. $$
For every $u\in\mathcal{E}_{\tau,\mu}(\tilde{C}_{0})$ we have that $\|u\|_{\ast}\leq \tilde{C}_{0}. $ Hence $\mathcal{E}_{\tau,\mu}(\tilde{C}_{0})$ is a bounded convex subset of $\mathcal{E}$. Furthermore, since the convergence in $\mathcal{E}$ implies the pointwise convergence, then $\mathcal{E}_{\tau,\mu}(\tilde{C}_{0})$ is a closed, bounded, and convex subset of $\mathcal{E}$. Furthermore, a sequence of functions in $\mathcal{E}_{\tau,\mu}(\tilde{C}_{0})$ converges with respect to norm $\|\cdot\|_{\ast}$ if and only if it  converges locally uniformly on $\R$.

We prove that the mapping $\mathcal{E}_{\tau,\mu}(\tilde{C}_{0})\ni u\mapsto U(\cdot;u)\in\mathcal{E}_{\tau,\mu}(\tilde{C}_{0})$ has a fixed point. We divide the proof in three steps.

\smallskip

\noindent {\bf Step 1.} In this step, we prove that the mapping $\mathcal{E}_{\tau,\mu}(\tilde{C}_{0})\ni u\mapsto U(\cdot;u)\in \mathcal{E}_{\tau,\mu}(\tilde{C}_{0})$ is compact.

 Let $\{u_{n}\}_{n\geq 1}$ be a sequence of elements of $\mathcal{E}_{\tau,\mu}(\tilde{C}_{0})$. Since $U(\cdot;u_{n})\in \mathcal{E}_{\tau,\mu}(\tilde{C}_{0})$ for every $n\geq 1$ then $\{U(\cdot;u_{n})\}_{n\geq 1}$ is clearly uniformly bounded by $\tilde{C}_{0}$. Using inequality \eqref{Eq_Convergence01}, we have that
\begin{equation*}
\sup_{t\geq 1}\|U(\cdot,t; U_{\mu}^+,u_{n})\|_{X^{\beta}}\leq M_{1}
\end{equation*}
for all $n\geq 1$ where $M_{1}$ is given by \eqref{Eq_Conv02}. Therefore there is $0<\nu\ll 1$ such that
\begin{equation}\label{Proof-MainTh3- Eq1}
\sup_{t\geq 1}\|U(\cdot,t; U_{\mu}^+,u_{n})\|_{C^{\nu}_{\rm unif}(\R)}\leq \tilde{M_{1}}
\end{equation} for every $n\geq 1$ where $\tilde{M_{1}}$ is a constant depending only on $M_{1}$. Since for every $n\geq 1$ and every $x\in\R$, we have that $U(x,t; U_{\mu}^+,u_{n})\to U(x;u_{n})$ as $t\to \infty,$ then it follows from \eqref{Proof-MainTh3- Eq1} that
\begin{equation}\label{Prof-MainTh3- Eq2}
\|U(\cdot;u_{n})\|_{C^{\nu}_{\rm unif}}\leq \tilde{M_{1}}
\end{equation} for every $n\geq 1$. Which implies that the sequence $\{U(\cdot;u_{n})\}_{n\geq 1}$ is equicontinuous. The Arzela-Ascoli's Theorem implies that there is a subsequence $\{U(\cdot;u_{n'})\}_{n\geq 1}$ of the sequence $\{U(\cdot;u_{n})\}_{n\geq 1}$ and a function $U\in C(\R)$ such that $\{U(\cdot;u_{n'})\}_{n\geq 1}$ converges to $U$ locally uniformly on $\R$. Furthermore, the function $U$ satisfies inequality \eqref{Prof-MainTh3- Eq2}. Combining this with the fact  $U_{\mu}^{-}(x)\leq U(x;u_{n'})\leq U_{\mu}^{+}(x)$ for every $x\in\R$ and $n\geq 1$, by letting $n$ goes to infinity, we obtain that  $U\in \mathcal{E}_{\tau,\mu}(\tilde{C}_{0})$. Hence  the mapping $\mathcal{E}_{\tau,\mu}(\tilde{C}_{0})\ni u\mapsto U(\cdot;u)\in \mathcal{E}_{\tau,\mu}(\tilde{C}_{0})$ is compact.

\smallskip

\noindent{\bf Step 2.} In this step, we prove that the mapping $\mathcal{E}_{\tau,\mu}(\tilde{C}_{0})\ni u\mapsto U(\cdot;u)\in \mathcal{E}_{\tau,\mu}(\tilde{C}_{0})$ is continuous.
 This follows from the arguments used in the proof of Step 2, Theorem 3.1, \cite{SaSh2}

Now by Schauder's Fixed Point Theorem, there is $U\in\mathcal{E}_{\tau,\mu}(\tilde{C}_{0})$ such that $U(\cdot;U)=U(\cdot)$. Then
$(U(x),V(x;U))$ is a stationary solution of \eqref{Main-eq2} with $c=c_\mu$. It is clear that
$$
\lim_{x\to\infty}\frac{U(x)}{e^{- \mu x}}=1.
$$

\noindent {\bf Step 3.} We claim that
$$
\lim_{x\to -\infty}U(x)=\frac{a}{b}.
$$
For otherwise, we may assume that there is $x_n\to -\infty$ such that $U(x_n)\to \lambda\not =\frac{a}{b}$ as $n\to\infty$. Define $U_{n}(x)=U(x+x_{n})$ for every $x\in\R$ and $n\geq 1$. By the arguments of Lemma \ref{MainLem02},  there is a subsequence $\{U_{n'}\}_{n\geq 1}$ of $\{U_{n}\}_{n\geq 1}$ and a function  $U^*\in C_{\rm unif}^b(\R)$ such that $\|U_{n'}-U^*\|_{\ast}\to 0$ as $n\to \infty$. Moreover, $(U^*,V_1(\cdot;U^*),V_2(\cdot;U^*))$ is also a stationary solution of \eqref{Main-eq2} with $c=c_\mu$.

\smallskip

\noindent {\bf Claim 1.} $\inf_{x\in\R}U^{*}(x)>0$.

 Indeed, let $0< \delta\ll 1$ be fixed. For  every $x\in\R$, there $N_{x}\gg 1$ such that $x+x_{n'}< x_{\delta}$ for all $n\geq N_{x}$. Hence, It follows from Remark \ref{Remark-lower-bound-for -solution} that  $$0< U_{\mu,\delta}^{-}(x_{\delta})\leq U(x+x_{n'}) \ \forall\ n\geq N_{x}.$$
Letting $n\to\infty$ in the last inequality, we obtain that $U_{\mu,\delta}^{-}(x_{\delta})\leq U^{*}(x)$ for every $x\in\R$. The claim thus follows.

\noindent {\bf Claim 2.}  $U^*(x)\equiv\frac{a}{b}$.

 Note also  that the function $(\tilde U(x,t),\tilde V_1(x,t),\tilde V_2(x,t))=(U^{*}(x-c_{\mu}t),V_1(x-c_{\mu}t,U^{*}),V_2(x-c_{\mu}t,U^{*}))$ solves \eqref{Main-eq2}. Then by Theorem B and Claim 1,
 $$
 \lim_{t\to\infty}\sup_{x\in\R}|U^*(x-c_\mu t)-\frac{a}{b}|=0.
 $$
 This implies that $U^*(x)=\frac{a}{b}$ for any $x\in\R$ and the claim thus follows.

 By Claim 2, we have $\lim_{n\to \infty} U(x_n)=U^*(0)=\frac{a}{b}$, which contracts to $\lim_{n\to \infty} U(x_n)=U^*(0)=\lambda\not =\frac{a}{b}$.
\end{proof}
%As a direct consequence of Theorem \ref{existence-tv-thm} we present the proof of Theorem A .

 Now, we are ready to prove Theorem C.

 \begin{proof}[Proof of Theorem C (i)]
Let us set
\begin{align*}
m_{\tau}(\chi_1,\mu_1,\lambda_1,\chi_2,\mu_2,\lambda_2):=\inf&\Big\{ \max\{2(\overline{M}+\tau c_{\mu}K-\frac{\chi_2\mu_2-\chi_1\mu_1}{2})\ ,\ (\tau c_{\mu}+\mu)K_{\tau,\mu}+\overline{M}_{\tau,\mu}, \}\,  : \nonumber\\
&\  \ 0<\mu<\min\{\sqrt{a}, \sqrt{\frac{\lambda_1+\tau a}{(1-\tau)_{+}}},\sqrt{\frac{\lambda_2+\tau a}{(1-\tau)_{+}}}\} \Big\}.
\end{align*}
Note that under the assumption of Theorem A, we have that
$$
b+\chi_2\mu_2-\chi_1\mu_1>m_{\tau}(\chi_1,\mu_1,\lambda_1,\chi_2,\mu_2,\lambda_2).
$$
Hence the open set
\begin{align*}
O_{\tau}:=\Big\{& \mu\in \Big(0\ ,\ \min\{\sqrt{a}, \sqrt{\frac{\lambda_1+\tau a}{(1-\tau)_{+}}},\sqrt{\frac{\lambda_2+\tau a}{(1-\tau)_{+}}}\}\Big)\ {\rm such\ that} \cr
&  \  \max\{2(\overline{M}+\tau c_{\mu}K+\frac{\chi_2\mu_2-\chi_1\mu_1}{2})\ ,\ (\tau c_{\mu}+\mu)K_{\tau,\mu}+\overline{M}_{\tau,\mu}, \}< b+\chi_2\mu_2-\chi_1\mu_1\Big\}
\end{align*}
is nonempty.
Let $(\mu^{**}_{\tau}, \mu^{*}_{\tau})$ denotes the right maximal open connected component of the open set $O_{\tau}$. That is $(\mu^{**}_{\tau}, \mu^{*}_{\tau})$ is the open connected component of $O_{\tau}$ with maximal length and, closed to $\sqrt{a}$.  Since the open interval $ \Big(0\ ,\ \min\{\sqrt{a}, \sqrt{\frac{\lambda_1+\tau a}{(1-\tau)_{+}}},\sqrt{\frac{\lambda_2+\tau a}{(1-\tau)_{+}}}\}\Big)$ has finite length, then $(\mu^{**}_{\tau}, \mu^{*}_{\tau}) $ is well defined.  Next, we set
\begin{align*}
c^{*}(\tau,\chi_1,\mu_1,\lambda_1,\chi_2,\mu_2,\lambda_2):=c_{\mu_{\tau}^{*}}, \quad {\rm and}\quad c^{**}(\tau,\chi_1,\mu_1,\lambda_1,\chi_2,\mu_2,\lambda_2):=c_{\mu^{**}_{\tau}}
\end{align*}where $c_{\mu}=\mu+\frac{a}{\mu}$.   Let $ c^{*}(\tau,\chi_1,\mu_1,\lambda_1$, $\chi_2,\mu_2,\lambda_2)< c<c^{**}(\tau,\chi_1,\mu_1,\lambda_1$, $\chi_2,\mu_2,\lambda_2)$ be given and let $\mu\in(\mu^{**}, \mu^*)$ be the unique solution of the equation $c_{\mu}=c$. It follows from Theorem \ref{existence-tv-thm}, that  \eqref{Main-eq01} has a traveling wave solution $(U(x,t),V_{1}(x,t),V_2(x,t))=(U(x-ct),V_1(x-ct),V_2(x-ct))$ with speed c connecting $(\frac{a}{b},\frac{a\mu_1}{b\lambda_1},\frac{a\mu_2}{b\lambda_2})$ and $(0,0,0)$. Moreover $\lim_{z\to\infty}\frac{U(z)}{e^{-\mu z}}=1$.

Observe that for every $\lambda_i,\mu_i>0$, $i=1,2$ and $\tau>0$ we have
$$
\lim_{(\chi_1,\chi_2)\to (0^+,0^+))}m_{\tau}(\chi_1,\mu_1,\lambda_1,\chi_2,\mu_2,\lambda_2)=0.
$$ 
 Thus we have that 
 $$  \lim_{(\chi_1,\chi_2)\to (0^+,0^+))}\mu^{*}_{\tau}=\min\{\sqrt{a}, \sqrt{\frac{\lambda_1+\tau a}{(1-\tau)_{+}}},\sqrt{\frac{\lambda_2+\tau a}{(1-\tau)_{+}}}\} 
 $$
and
$$\lim_{(\chi_1,\chi_2)\to (0^+,0^+))}\mu^{**}_{\tau}=0. $$ Thus
$$
\lim_{(\chi_1,\chi_2)\to (0^+,0^+))}c^{**}(\tau,\chi_1,\mu_1,\lambda_1,\chi_2,\mu_2,\lambda_2)=\infty,$$
and
$$\lim_{(\chi_1,\chi_2)\to (0^+,0^+))}c^{*}(\tau,\chi_1,\mu_1,\lambda_1,\chi_2,\mu_2,\lambda_2)=c_{\tilde{\mu}^*},
$$
where $\tilde{\mu}^*={\min\{\sqrt{a}, \sqrt{\frac{\lambda_1+\tau a}{(1-\tau)_{+}}},\sqrt{\frac{\lambda_2+\tau a}{(1-\tau)_{+}}}\}}$.
\end{proof}

\medskip

Next, we present the proof of Theorem C (ii).

\medskip

\begin{proof}[Proof of Theorem C (ii)]
 For every $c_n>c^*$ with $c_n\to c^*=c^{*}(\tau,\chi_1,\mu_1,\lambda_1,\chi_2,\mu_2,\lambda_2)$, let $(U^{c_n}(x),V_1^{c_n}(x),V_2^{c_n}(x))$ be the traveling wave solution solution of \eqref{Main-eq1} connecting  $(\frac{a}{b},\frac{a\mu_1}{b\lambda_1},\frac{a\mu_2}{b\lambda_2})$  and $(0,0,0)$ with speed $c_n$ given by Theorem C (i). For each $n\geq 1$, note that the set $\{x\in\R\ :\ U^{c_n}(x_n)=\frac{a}{2b}\}$ is compact  and nonempty, so there is $x_n\in\R$ such that
 $$x_n=\min\{x\in\R\ :\ U^{c_n}(x_n)=\frac{a}{2b}\}.$$  Since $\sup_n\|U^{c_n}\|_{\infty}<\infty$, hence by estimates for parabolic equations, without loss of generality, we may suppose that that $U^{c_n}(x+x_n)\to U^*(x)$ as $n\to\infty$ locally uniformly. Furthermore, taking $V_i^*(x)=V_i(x,U^*)$, $i=1,2$, it holds that $(U^*,V_1^*,V_2^*)$ solves
 \begin{equation}\label{kk-2}
 \begin{cases}
0=U^*_{xx}+((c^*+(\chi_2V_2^*-\chi_1V_1^*)_x)U^*)_x+U^*(a-bU^*),\quad x\in\R\cr
0=\partial_{xx}V_1^*+\tau\partial_xV_1^*-\lambda_1 V^*+\mu_1 U^*,\quad x\in\R,\cr
0=\partial_{xx}V_2^*+\tau\partial_xV_2^*-\lambda_2 V^*+\mu_2 U^*,\quad x\in\R,
\end{cases}
 \end{equation}
 $U^*(0)=\frac{a}{2b}$ and $U^*(x)\geq \frac{a}{2b}$ for every $x\leq 0$. Next we claim that
 \begin{equation}\label{kk-3}
\overline{U}^*(\infty):= \limsup_{x\to\infty}U^*(x)=0.
 \end{equation}
 Suppose on the contrary that \eqref{kk-3} does not hold. Then there is a sequence $\{z_n\}_{n\geq1}$ such that $z_n<z_{n+1}$ for every $n$, $z_n\to\infty$, $z_1=0$, and
 \begin{equation*}
\inf_{n\ge 1} U^*(z_n)>0.
 \end{equation*}
 For every $n\geq 1$ let $\{y_n\}_{n\geq1}$ be the sequence defined by
 $$
U^*(y_n)=\min\{U^*(x)\ :\ z_n\leq x\leq z_{n+1}\}.
 $$
 Observe  that
 $$
\lim_{n\to\infty}U^*(y_n)=\inf_{x\in\R}U^*(x).
 $$
 Since $(U^*(x-c^*t),V_1^*(x-c^*t),V_2^*(x-c^*t))$ is a positive entire solution of \eqref{Main-eq1} with
 $$
\liminf_{x\to-\infty}U^*(x) \geq \frac{a}{2b}\quad \text{and} \quad U^*(0)\neq \frac{a}{b},
 $$
 then by the stability of the constant equilibrium $(\frac{a}{b},\frac{a\mu_1}{b\lambda_1},\frac{a\mu_2}{b\lambda_2})$ provided by Theorem B, we obtain that
 $$
0=\inf_{x\in\R}U^*(x)=\lim_{n\to\infty}U^*(y_n).
 $$
 Therefore, without loss of generality, we may suppose that $y_n\in(x_n,x_{n+1})$ for every $n\geq 1$ with
 $$
\frac{d^2}{dx^2}U^*(y_n)\geq 0 \quad \text{and}\quad  \frac{d}{dx}U^*(y_n)=0, \quad \forall\ n\geq1.
 $$
 Note also that
 $$
\lim_{n\to\infty}V_i^*(y_n)=0\quad \text{and}\quad  \lim_{n\to\infty}\partial_xV_i^*(y_n)=0 \quad i=1,2.
 $$
 Thus for $n$ large enough, we have that
 $$
 U_{xx}^*(y_n)-((c^*+(\chi_2V_2-\chi_1V_1^*)_x)U^*)_x(y_n)+U^*(y_n)(a-bU^*(y_n))>0,
 $$
 which contradicts to \eqref{kk-2}. Therefore, \eqref{kk-3} holds.

 It follows again from
 $$
\liminf_{x\to-\infty}U^*(x) \geq \frac{a}{2b} $$  and the stability of the constant equilibrium $(\frac{a}{b},\frac{a\mu_1}{b\lambda_1},\frac{a\mu_2}{b\lambda_2})$ provided by Theorem B that $$ \liminf_{x\to-\infty}U^*(x)=\frac{a}{b}.$$
 Therefore $(u^{c^*}(t,x),v^{c^*}(t,x))=(U^*(x-c^*t),V^*(x-c^*t))$ is a traveling wave solution of \eqref{Main-eq1} with speed $c^*$ connecting $(\frac{a}{b},\frac{a\mu_1}{b\lambda_1},\frac{a\mu_2}{b\lambda_2})$  and $(0,0,0)$.

\end{proof}

For clarity sake in the proof of Theorem C (iii), we first present the following lemma. 

\medskip

\begin{lem}\label{decreasing at infinity}
Let $(u,v_1,v_2)(x,t)=(U(x-ct),V_1(x-ct),V_2(x-ct))$ be traveling wave solution of \eqref{Main-eq1} connecting $(\frac{a}{b},\frac{a\mu_1}{b\lambda_1},\frac{a\mu_2}{b\lambda_2})$ and $(0,0,0)$ with speed $c$.  There is $X_0\gg 1$ such that 
\begin{equation}\label{zz-eq1}
U \ :\  [X_0,\infty)\to (0,\infty) \quad \text{is monotone decreasing.}
\end{equation} 

\end{lem}
\begin{proof}Suppose by contradiction that \eqref{zz-eq1} does not hold. There exists a sequence $\{x_n\}_{n\geq 1}$ with $x_n\to\infty$ such that $x_n$ is a local minimum point of $U$. Hence,
$$ 
U_{xx}(x_n)\geq 0 \quad \text{and}\quad  U_{x}(x_n)=0,\quad \forall\ n\geq 1.
$$
Since $U(\infty)=0$, then it holds that 
$$ 
\lim_{n\to\infty}U(x_n)=\lim_{n\to\infty}V_i(x_n)=\lim_{n\to\infty}\partial_xV_i(x_n)=0.
$$
Therefore, there is $n\gg 1$ such that 
\begin{align*} 
0<&U_{xx}(x_n)+(c+(\chi_2V_2-\chi_2V_2)_x(x_n))U_{x}(x_n)\cr
&+(a+(\chi_2\lambda_2V_2-\chi_1\lambda_1V_1)(x_n)+c\tau(\chi_1V_1-\chi_2V_2)_x(x_n)-(b+\chi_2\mu_2-\chi_1\mu_1)U(x_n))U(x_n)
\end{align*}
which contradicts to  \eqref{Eq_MainLem02}. Thus, \eqref{zz-eq1} holds. 

\end{proof}

\medskip

Next, we present the proof of Theorem C(iii).

\medskip

 \begin{proof}[Proof of Theorem C (iii)] Let $(u,v_1,v_2)(x,t)=(U(x-ct),V_1(x-ct),V_2(x-ct))$ be traveling wave solution of \eqref{Main-eq1} connecting $(\frac{a}{b},\frac{a\mu_1}{b\lambda_1},\frac{a\mu_2}{b\lambda_2})$ and $(0,0,0)$ with speed $c$.  Let $X_0\gg 1$ be given by Lemma \ref{decreasing at infinity}. We shall show that $c\geq 2\sqrt{a}$. Suppose by contradiction that $c<2\sqrt{a}$. Choose $q\in(\max\{c,0\},2\sqrt{a})$ and $0<\varepsilon\ll 1$ satisfying
 \begin{equation*}
 q+\varepsilon<2\sqrt{a-\varepsilon}.
\end{equation*} 
Since $(U,V_1,V_2)(\infty)=(0,0,0)$, there is $X_{\varepsilon}> X_0$ such that 
 $$ 
(1+\tau |c|)|(\chi_2V_2+\chi_1V_1)_x|(x)+|(\chi_1\lambda_1V_1-\chi_2\lambda_2V_2)|(x) +|b-\chi_1\mu_1+\chi_2\mu_2|U(x)<\varepsilon,\quad \forall\ x>X_{\varepsilon}.  
 $$
 Hence, since $U'(x-ct)\leq 0$ for $x-ct\geq X_{\varepsilon}$, the function $u(x,t)=U(x-ct)$ satisfies
  \begin{align}\label{zz-eq2}
 u_t\geq  u_{xx} +(\chi_2v_2-\chi_1v_1)_xu_x +(a-\varepsilon )u
 \geq  u_{xx}+\varepsilon u_{x} +(a-\varepsilon )u, \quad\ x\ge ct+ X_{\varepsilon}.
 \end{align}
Observe that with $L:=\frac{2\pi}{\sqrt{4(a-\varepsilon)-(q+\varepsilon)^2}}$ the function 
$$ 
\underline{u}(x,t)=m_0e^{-\frac{q+\varepsilon}{2}(x-qt)}\sin(\frac{\pi}{L}(x-qt)), \quad qt\leq x\leq qt+ L, \ t\in\R,
$$
with $m_0=\min\{U(x-ct_0) \ :\ 0\leq x\leq L+qt_0 \}$ where $t_0:=\frac{X_{\varepsilon}}{q-c}$
satisfies $0\leq \underline{u}(t,x)\leq m_0$,
\begin{equation}\label{zz-eq3} 
\underline{u}_t= \underline{u}_{xx}+\varepsilon\underline{u}_x+(a-\varepsilon)\underline{u},
\end{equation}
\begin{equation}\label{zz-eq4} \underline{u}(qt,t)=0< u(qt,t)\quad \text{and}\quad  \underline{u}(qt+L,t)=0<u(qt+L,t), \quad \forall\ t\in\R,
\end{equation}
and 
\begin{equation}\label{zz-eq5} 
\underline{u}(x,t_0)\leq m_0\leq U(x-ct_0)=u(x,t_0),\quad \forall qt_0\leq x\leq qt_0+L. 
\end{equation}
Therefore, by comparison principle for parabolic equations, it follows from \eqref{zz-eq2}-\eqref{zz-eq5} that 
$$ 
\underline{u}(t,x)\leq u(t,x)=U(x-ct), \quad \forall qt\leq x\leq qt + L, \ t\geq t_0.
$$
In particular, for $x=qt+\frac{L}{2}$, $t\ge t_0$, we have from the last inequality that 
$$ 
0<m_0e^{-\frac{(q+\varepsilon)L}{4}}=\underline{u}(qt+\frac{L}{2},t)\leq u(qt+\frac{L}{2},t)= U((q-c)t+\frac{L}{2}), \quad  t\ge t_0. 
$$
Thus $q\leq c$, since $U(\infty)=0$.  Which contradicts to $q>c$. Therefore we must have $c\geq 2\sqrt{a}$. 
\end{proof}

\end{document}